\documentclass[11pt,longbibliography]{amsart}

\usepackage[utf8]{inputenc}

\usepackage[T1]{fontenc}

\usepackage[a4paper, margin=2.7cm]{geometry}

\usepackage{amsthm}
\usepackage{amsmath}
\usepackage{geometry}
\usepackage{amssymb}
\usepackage{subcaption}
\usepackage{enumerate}
\usepackage{xcolor}
\usepackage{graphicx}
\usepackage{float}
\usepackage[hang,flushmargin]{footmisc} 
\usepackage[font={small}]{caption}
\usepackage{url}
\newtheorem{theorem}{Theorem}[section]
\newtheorem{corollary}{Corollary}[theorem]
\newtheorem{proposition}[theorem]{Proposition}
\newtheorem{lemma}[theorem]{Lemma}
\newtheorem{definition}[theorem]{Definition}

\theoremstyle{definition}
\newtheorem{remark}[theorem]{Remark}
\theoremstyle{definition}
\newtheorem{example}[theorem]{Example}
\DeclareMathOperator\supp{supp}
\usepackage{mathrsfs}           

\usepackage[colorlinks=true,linkcolor=blue,citecolor=blue,urlcolor=blue,breaklinks]{hyperref}

\usepackage{bbm} 
\let\svthefootnote\thefootnote
\newcommand\freefootnote[1]{%
  \let\thefootnote\relax%
  \footnotetext{#1}%
  \let\thefootnote\svthefootnote%
}

\usepackage{breakurl}
\usepackage{url}

\title[Clark measures for product functions and multiplicative embeddings]{Clark measures on polydiscs associated to product functions and multiplicative embeddings}
\author{Nell P. Jacobsson}

\AtEndDocument{\bigskip{\footnotesize
  \noindent
  \textsc{Department of Mathematics, Stockholm University, 106 91 Stockholm, Sweden.} \\
  \textit{E-mail address}: \texttt{\href{mailto:nell.jacobsson@math.su.se}{nell.jacobsson@math.su.se}}
}}

\begin{document}
\maketitle
\begin{abstract}
We study Clark measures on the unit polydisc, giving an overview of recent research and investigating the Clark measures of some new examples of multivariate inner functions. In particular, we study the relationship between Clark measures and multiplication; first by introducing compositions of inner functions and multiplicative embeddings, and then by studying products of one-variable inner functions. 
\end{abstract}


\section{Introduction}
\label{sec:intro}

\freefootnote{\textit{Date:} September 8, 2023.}
\freefootnote{2020 \textit{Mathematics Subject Classification.} 28A25, 28A35 (primary); 32A10, 30J05 (secondary).}
\freefootnote{\textit{Key words and phrases.} Clark measure, multiplicative embedding, product function.}


Let 
$$ \mathbb{D}^d := \{ (z_1, z_2, \ldots, z_d) \in \mathbb{C}^d: |z_j| < 1, \quad j = 1,2,\ldots, d \}$$
denote the unit polydisc in $d$ variables, and 
$$ \mathbb{T}^d := \{ (\zeta_1, \zeta_2, \ldots, \zeta_d) \in \mathbb{C}^d: |\zeta_j| = 1, \quad j = 1,2,\ldots, d \}$$
its distinguished boundary. Note that this is only a subset of the boundary $\partial \mathbb{D}^d$. For $d=2$, the set $ \mathbb{T}^2$ defines a two-dimensional torus. 

For $z \in \mathbb{D}^d$ and $\zeta \in \mathbb{T}^d$, we introduce the Poisson kernel on $\mathbb{D}^d$ as a product of one-variable Poisson kernels:
\begin{align*}
    P_z( \zeta) = P(z, \zeta) := \prod_{j=1}^d P_{z_j}(\zeta_j) \quad \text{where} \quad P_{z_j}( \zeta_j) := \frac{1-|z_j|^2}{|\zeta_j-z_j|^2}. 
\end{align*}
Given a complex Borel measure $\mu$ on $\mathbb{T}^d$, we define its Poisson integral as
$$P[d \mu](z) := \int_{\mathbb{T}^d} P(z,\zeta) d \mu(\zeta), \quad z \in \mathbb{D}^d.
$$

If $\phi: \mathbb{D}^d \to \mathbb{D}$ is a bounded holomorphic function, then
\begin{align*}
    \Re \bigg( \frac{\alpha + \phi(z)}{\alpha- \phi(z)} \bigg) = \frac{1-|\phi(z)|^2}{|\alpha-\phi(z)|^2}
\end{align*}
is positive and pluriharmonic (i.e. locally the real part of an analytic function) on $\mathbb{D}^d$. Hence, by Herglotz' theorem, there exists a unique positive Borel measure $\sigma_\alpha$ on $\mathbb{T}^d$ such that
\begin{align*}
    \frac{1-|\phi(z)|^2}{|\alpha-\phi(z)|^2} = P[d \sigma_\alpha](z) = \int_{\mathbb{T}^d} P(z,\zeta) d \sigma_\alpha(\zeta).
\end{align*}
We call $\{ \sigma_\alpha \}_{\alpha \in \mathbb{T}}$ the \textit{Aleksandrov-Clark measures} associated to $\phi$. A measure whose Poisson integral is the real part of an analytic function is generally called an RP-measure or pluriharmonic measure. 

Recall that for a function $f: \mathbb{D} \to \mathbb{C}$ and some point $\zeta \in \mathbb{T}$, we say that $f(z)$ approaches $L \in \mathbb{C}$ non-tangentially, denoted
$$ L := \angle \lim_{z \to \zeta} f(z),
$$
if $f(z) \to L$ whenever $z \to \zeta$ in every fixed Stolz domain
\begin{align*}
    \Gamma_\alpha ( \zeta) := \{ z \in \mathbb{D}: |z- \zeta| < \alpha ( 1-|z|) \}, \quad \alpha > 1.
\end{align*}
This notion extends to multivariate functions: let 
$f: \mathbb{D}^d \to \mathbb{C}$ and $\zeta \in \mathbb{T}^d$. We say that $f$ has a non-tangential limit $L$ at $\zeta$ if $f(z) \to L$ as $z \to \zeta$, where each $z_j \to \zeta_j$ non-tangentially in the one-variable sense. 

Let $m_d$ denote the normalized Lebesgue measure on $\mathbb{T}^d$, where we write $m$ instead of $m_1$ when $d=1$. For any bounded holomorphic function $\phi: \mathbb{D}^d \to \mathbb{C}$, Fatou's theorem for polydiscs (see Chapter XVII, Theorem 4.8 in \cite{Zygmund}) ensures that the non-tangential limits
\begin{align*}
    \phi^* ( \zeta) = \angle \lim_{\mathbb{D}^d \ni z \to \zeta} \phi ( z) 
\end{align*}
exist for $m_d$-almost every $\zeta \in \mathbb{T}^d$. We say that $\phi: \mathbb{D}^d \to \mathbb{D}$ is an \textit{inner} function if it is bounded, holomorphic and $|\phi^*(\zeta)| = 1$ for $m_d$-almost every $\zeta \in \mathbb{T}^d$. If $\phi$ is inner, we call $\{ \sigma_\alpha \}_{\alpha \in \mathbb{T}}$ the \textit{Clark measures} of $\phi$ instead.

In one variable, there are detailed descriptions of Clark measures and their behavior, see e.g. Chapter 11 of \cite{Intro}. For instance, it is known that to each inner function on the unit disc, one can associate not only a family of Clark measures, but also a family of unitary operators. In \cite{Doubtsov}, Doubtsov studies the families of Clark measures and operators associated to any given inner function on $\mathbb{D}^d$, and successfully extends some classical Clark theory to the multivariate setting. For an inner function $\phi$ on $\mathbb{D}^d$, we define the model space associated to $\phi$ as
    $$K_\phi := H^2(\mathbb{T}^d) \ominus \phi H^2(\mathbb{T}^d)$$ 
    and operators 
    $$ T_\alpha : K_\phi \to L^2(\sigma_\alpha), \quad \alpha \in \mathbb{T},
    $$
    In \cite{clark_og}, Clark proves these operators to be unitary when $d=1$, and in \cite{Doubtsov}, Doubtsov investigates what happens in higher dimensions. Let 
    $$ C_w(z) := \prod_{j=1}^d \frac{1}{1-z_j \overline{w_j}}, \quad z_j \in \overline{\mathbb{D}}^d, w_j \in \mathbb{D}^d
    $$
    denote the Cauchy kernel in $d$ variables. Doubtsov then defines the operator $T_\alpha$ on reproducing kernels $K_w(z)$ as 
    $$ T_\alpha [K_w](\zeta) := (1- \alpha \overline{\phi(w)}) C_w(\zeta), \quad \zeta \in \mathbb{T}^d, w \in \mathbb{D}^d,
    $$ 
    and later extends this to all functions in $K_\phi$ using density. In his main result (Theorem 3.2, \cite{Doubtsov}), he shows that $T_\alpha$ being a unitary operator is equivalent to the polydisc algebra $A(\mathbb{D}^d)$ being dense in $L^2(\sigma_\alpha)$. This work will focus on the measure-theoretic aspects of Clark theory, but the reader can find notes on Clark embedding operators for rational inner functions in particular in Section 4 of \cite{clark2}. 

    The case of rational inner functions is thoroughly investigated in both \cite{clark2} and \cite{clark1}, where the authors explicitly characterize the Clark measures of general rational inner functions of bidegree $(m,n)$. This work aims to extend this theory to other classes of inner functions, constructed from one-variable functions. In one variable, any inner function can be expressed as the product of a Blaschke product, a monomial and a singular inner function (Theorem 2.14, \cite{Intro}), but there is no such simple general structure known for inner functions in higher dimensions. Hence, it is much harder to say anything about the general case for $d>1$. By studying multivariate functions constructed from one-variable inner functions, we can draw on the one-dimensional properties.

    As an aside, we note that there are two natural settings for multivariate Clark theory; one could either study Clark measures on the unit polydisc $\mathbb{D}^d$ or the unit $d$-ball. Clark theory on the unit $d$-ball has been explored in detail in e.g. \cite{doubtsov_sphere}. We restrict ourselves to $\mathbb{D}^d$ in this work, and mainly dimension $d=2$. 

\subsection{Overview} First, in Section \ref{sec:prel}, we introduce some basic theory concerning Clark measures in $\mathbb{T}^d$. Next, we survey some notable properties of Clark measures in one variable. In Section \ref{sec:rifs}, we give an overview of recent progress concerning rational inner functions (RIFs). In particular, we present results from \cite{clark2} --- for example, we will see that the Clark measures of bivariate RIFs are supported on finite unions of analytic curves, and that one can characterize their behavior along these curves. 

In Section \ref{sec:mult}, we consider the multiplicative embedding
$$ \Phi(z):= \phi(z_1 z_2 \cdots z_d), \quad z \in \mathbb{D}^d
$$
which produces a multivariate inner function given an inner function $\phi$ in one variable. First we investigate the case $d=2$, where we characterize the unimodular level sets of $\Phi$ and see that these can always be parameterized by --- potentially infinitely many --- antidiagonals in $\mathbb{T}^2$. We conclude this section by presenting a concrete structure formula for the Clark measures associated to $\Phi$ in $d$ variables. In Section \ref{sec:prod}, we turn to product functions 
$$ \Psi(z) := \phi(z_1) \psi(z_2), \quad z \in \mathbb{D}^2
$$
for inner functions $\phi$ and $\psi$. We then prove a structure formula for the Clark measures of $\Psi $ under certain assumptions on $\phi$ and $\psi$. 

Throughout the text, we present examples of bivariate inner functions with explicit characterizations of their Clark measures. Finally, in Section \ref{sec: close}, we discuss possible further research tied to our results and raise some open questions. 


\section{Preliminaries}
\label{sec:prel}
\subsection{Elementary Clark theory in polydiscs} If $\phi$ is an inner function with associated Clark measure $\sigma_\alpha$, then
$$ P[d \sigma_\alpha](z) = \frac{1-|\phi(z)|^2}{|\alpha-\phi(z)|^2} = 0 \quad \text{ $m_d$-almost everywhere on $\mathbb{T}^d$}.$$ 
Clearly, the numerator goes to zero $m_d$-almost everywhere. As $\phi - \alpha$ is bounded, it lies in the Hardy space $H^2( \mathbb{T}^d) \subset N(\mathbb{T}^d)$; then Theorem 3.3.5 in \cite{Rudin} states that $\log( \phi^* - \alpha) $ lies in $L^1(\mathbb{T}^d)$. This in turn implies that $\phi^* - \alpha$ must be non-zero $m_d$-almost everywhere on $\mathbb{T}^d$. Hence, $P[d \sigma_\alpha](z) = 0$ $m_d$-almost everywhere on $\mathbb{T}^d$, as asserted. 

A notable consequence of this result is that if $\phi$ is an inner function, then its Clark measures $\{ \sigma_\alpha \}_{\alpha \in \mathbb{T}}$ must be singular with respect to the Lebesgue measure. To see this, we decompose $\sigma_\alpha$ into an absolutely continuous measure $\tau^1_\alpha = f_\alpha m_d$, $f_\alpha \in L^1(\mathbb{T}^d)$, and a $m_d$-singular measure $\tau^2_\alpha$ (see Theorem 6.10, \cite{RC_Rudin}). Then Theorem 2.3.1 in \cite{Rudin} states that the function 
$$u(z):= P[d\sigma_\alpha ](z) = P[f_\alpha dm_d + d\tau^2_\alpha](z) 
$$
satisfies $u^*(\zeta) = f_\alpha(\zeta)$ for $m_d$-almost every $\zeta \in \mathbb{T}^d$. However, we saw already that $P[d\sigma_\alpha] = 0$ $m_d$-almost everywhere on $\mathbb{T}^d$; hence $f_\alpha = 0$ $m_d$-almost everywhere on $\mathbb{T}^d$. We can thus conclude that $\sigma_\alpha$ is a $m_d$-singular measure for each $\alpha \in \mathbb{T}$. Moreover, as asserted in \cite{Doubtsov}, two Clark measures $\sigma_\alpha$ and $\sigma_\beta$ associated to an inner function $\phi$ are mutually singular whenever $\alpha \neq \beta$. 

Observe that since
$$\int_{\mathbb{T}^d} d \sigma_\alpha = \int_{\mathbb{T}^d} P (0,\zeta) d \sigma_\alpha =  \frac{1-|\phi(0)|^2}{|\alpha-\phi(0)|^2} < \infty, $$
the measure $\sigma_\alpha$ is finite for each $\alpha \in \mathbb{T}$. In particular, $\sigma_\alpha$ is a probability measure if the associated inner function $\phi$ satisfies $\phi(0)=0$.

For an inner function $\phi$ and a constant $\alpha \in \mathbb{T}$, we define the so called \textit{unimodular level set}
\begin{align*}
    \mathcal{C}_\alpha( \phi ) := \text{Clos} \Bigl \{ \zeta \in \mathbb{T}^d : \lim_{r \to 1^-} \phi( r \zeta ) = \alpha \Bigl \},
\end{align*}
where the closure is taken with respect to $\mathbb{T}^d$. The following proposition is a generalization of Lemma 2.1 in \cite{clark2}, where it is proven for rational inner functions.
\begin{proposition} Let $\phi : \mathbb{D}^d \to \mathbb{C}$ be an inner function, and let $\alpha$ be a unimodular constant. Then $\supp (\sigma_\alpha) \subset \mathcal{C}_\alpha ( \phi) $.
\label{support}
\end{proposition}
With the exception of some details, the proof uses the same arguments as in \cite{clark2}. We include it for the interested reader. 
\begin{proof} Let $B \subset \mathbb{T}^d$ be an open ball such that $\lim_{r \to 1^-} \phi(r \zeta) \neq \alpha$ for all $\zeta \in B$. Our goal is to show that $\sigma_\alpha ( B ) = 0.$ Recall that the Poisson kernel is positive; hence, 
\begin{align*}
    \int_B P(r \zeta, \eta) d \sigma_\alpha(\eta) \leq  \int_{\mathbb{T}^d} P(r \zeta, \eta) d \sigma_\alpha(\eta) = \frac{1-|\phi(r \zeta)|^2}{|\alpha-\phi(r \zeta)|^2}
\end{align*}
for all $\zeta \in B$ and every $0 \leq r < 1$. We make two observations now: first of all, we note that since $\phi$ is inner, the right-hand side tends to zero for $m_d$-almost every $\zeta \in B$ as $r \to 1^-$. So
\begin{align*}
    \lim_{r \to 1^-} \int_B P(r \zeta, \eta) d \sigma_\alpha(\eta) = 0 \quad \text{$m_d$-almost everywhere in $B$}.
\end{align*}
 Secondly, since $\phi$ is bounded on the unit polydisc and $\phi(r \zeta) \not \to \alpha$ on $B$, we have that 
\begin{align}
    \limsup_{r \to 1^-} \int_B P(r \zeta, \eta) d \sigma_\alpha(\eta) \leq  \limsup_{r \to 1^-}  \frac{1-|\phi(r \zeta)|^2}{|\alpha-\phi(r \zeta)|^2} < \infty
    \label{fin}
\end{align}
for \textit{all} $\zeta \in B$. Here, we take the limit superior instead of the limit, as the limit of the right-hand side need not exist for every point in $B$. 

Now define 
\begin{align*}
    D_r ( \zeta ) := \{ \eta: |r \zeta_j - \eta_j | \leq 2(1-r): \quad j = 1, \ldots, d \}.
\end{align*}
For $\eta \in D_r(\zeta)$, one can show that
\begin{align*}
    \bigg( \frac{1+r}{4(1-r)} \bigg)^d \leq P(r \zeta, \eta) 
\end{align*}
(see details on p. 4, \cite{clark2}), and so 
\begin{align*}
     \bigg( \frac{1+r}{4(1-r)} \bigg)^d  \sigma_\alpha ( B \cap D_r ( \zeta )) \leq \int_{B \cap D_r ( \zeta) } P( r \zeta, \eta ) d \sigma_\alpha(\eta) \leq \int_{B } P( r \zeta, \eta ) d \sigma_\alpha(\eta), 
\end{align*}
which in turn implies that 
\begin{align}
    \lim_{r \to 1^-} \frac{\sigma_\alpha ( B \cap D_r ( \zeta )) }{(1-r)^d} = 0 \quad \text{$m_d$-almost everywhere in $B$}
    \label{zero}
\end{align}
and, moreover, that the limit superior of this quotient is finite for all $\zeta \in B$ by \eqref{fin}. 

In polar coordinates, we can express $D_r( \zeta)$ as
\begin{align*}
    D_r ( \zeta) = \biggl\{ \zeta e^{i \theta} : |\theta_j| < \cos^{-1}\bigg( 1-\frac{3(1-r)^2}{2r} \bigg), \quad j=1, \ldots, d \biggl\}
\end{align*}
(details on p. 5, \cite{clark2}). We observe that this is, as a subset of $\mathbb{T}^d$, a product of $d$ copies of the same interval. Hence, as $r \to 1^-$, we may estimate the Lebesgue measure of this set as 
\begin{align*}
    |D_r( \zeta) | = 2^d \cos^{-1} \bigg( 1-\frac{3(1-r)^2}{2r} \bigg)^d \geq c(d) \sqrt{\frac{3(1-r)^2}{2r}}^d \geq c'(d) (1-r)^d
\end{align*}
for constants $c(d), c'(d)$ dependent on $d$. Together with \eqref{zero}, this shows that
\begin{align*}
    \lim_{r \to 1^-} \frac{\sigma_\alpha ( B \cap D_r( \zeta)) }{|D_r( \zeta) |} = 0 \quad \text{$m_d$-almost everywhere in $B$,}
\end{align*}
and that the limit superior must be finite for all $\zeta \in   B$.

Note that per definition, $D_r( \zeta)$ is a $d$-dimensional cube with volume tending to zero as $r \to 1^-$ for every $\zeta \in B$. We now claim that 
\begin{align}
    \limsup_{r \to 1^-} \frac{\sigma_\alpha ( B \cap D_r( \zeta)) }{|D_r( \zeta) |} = 0 \quad \text{for every $\zeta \in B$.} 
    \label{limzero}
\end{align}
To prove this, suppose there exists some $z \in B$ such that the limit superior in \eqref{limzero} is nonzero. Since $\sigma_\alpha$ is a finite measure, we have that $\sigma_\alpha ( B \cap D_r( z )) < \infty$. Together with the fact that the denominator tends to zero, this would imply that the limit, and hence limit superior, is infinite for $z \in B$, which is a contradiction by our previous arguments. Hence \eqref{limzero} holds. 

Since $|D_r( \zeta) | \to 0 $ as $r \to 1^-$, the limit in \eqref{limzero} implies that the $n$-dimensional upper density of the restriction measure $( \sigma_\alpha)_{|B}$, defined as $(\sigma_\alpha)_{|B}(A) := \sigma_\alpha ( B \cap A)$, is zero for every point in $\mathbb{T}^d$ (see e.g. Proposition 2.2.2 in \cite{upperdensity}). Thus, $(\sigma_\alpha)_{|B}$ is equal to zero, which in turn implies that $\sigma_\alpha (B) = 0$. \end{proof} 
The inclusion in Proposition \ref{support} is not necessarily strict. In fact, for the classes of functions studied in this text, it follows from our structure formulas (Theorem \ref{generic}, \ref{exceptional} and Corollary \ref{zw_formula}) that $\supp(\sigma_\alpha) = \mathcal{C}_\alpha$. 

Finally, we record a fact which will be used in several proofs down the line:
\begin{lemma}
    The linear span of Poisson kernels $\mathcal{M} := \text{span}\{ P_z : z \in \mathbb{D}^d \}$ is dense in $C ( \mathbb{T}^d )$. 
    \label{density}
\end{lemma}
The proof is a straight-forward generalization of the proof of Proposition 1.17 in \cite{Intro}. 

\subsection{Clark measures in one variable}
Before getting into examples of Clark measures in higher dimensions, we quote some results from Clark theory in one variable. 
To formulate the main result, we must first introduce the concept of angular derivatives. 
\begin{theorem}
    For an analytic function $f$ on $\mathbb{D}$ and $\zeta_0 \in \mathbb{T}$, the following are equivalent: 
    \begin{enumerate}[(i)]
        \item The non-tangential limits 
        $$f ( \zeta_0 ) = \angle \lim_{z \to \zeta_0} f(z) \quad \text{and} \quad \angle \lim_{z \to \zeta_0} \frac{f(z)- f(\zeta_0)}{z-\zeta_0} $$
        exist; 
        \item The derivative function $f'$ has a non-tangential limit at $\zeta_0$.
    \end{enumerate}
    Under the equivalent conditions above, 
    $$ \angle \lim_{z \to \zeta_0 } \frac{f(z) - f(\zeta_0) }{z- \zeta_0} = \angle  \lim_{z \to \zeta_0 } f'(z).
    $$
    \label{theorem ang der}
\end{theorem}
\begin{proof}
    See Theorem 2.19 in \cite{Intro}. 
\end{proof}

\begin{definition} Assuming the conditions of the theorem, we call 
$$ f'( \zeta_0) := \angle \lim_{z \to \zeta_0 } \frac{f(z) - f(\zeta_0) }{z- \zeta_0} = \angle  \lim_{ z \to \zeta_0} f'(z) 
$$
the angular derivative of $f$ at $\zeta_0$. Furthermore, if $f$ maps $\mathbb{D}$ to itself, we say that $f$ has an angular derivative in the sense of Carathéodory at $\zeta_0 \in \mathbb{T}$ if $f$ has an angular derivative at $\zeta_0$ and $f( \zeta_0 ) \in \mathbb{T}$. 
\end{definition}

We now have the machinery needed to state the following proposition, which will be extremely useful to us in later sections:
\begin{proposition}
    Let $\phi: \mathbb{D} \to \mathbb{C}$ be an inner function and let $\alpha \in \mathbb{T}$. Then the associated Clark measure $\sigma_\alpha$ has a point mass at $\zeta \in \mathbb{T}$ if and only if 
    $$ \phi^* ( \zeta) = \lim_{r \to 1^-} \phi(r \zeta) = \alpha
    $$
    and $\phi$ has a finite angular derivative in the sense of Carathéodory at $\zeta$. In this case, 
    $$ \sigma_\alpha( \{ \zeta \} ) = \frac{1}{| \phi'(\zeta)|} < \infty \quad \text{and} \quad \phi'(\zeta) = \frac{\alpha \overline{\zeta}}{\sigma_\alpha( \{ \zeta \} )}.
    $$
    \label{one_var pointmass}
\end{proposition}
\begin{proof}
    See Proposition 11.2 in \cite{Intro}.
\end{proof}

\begin{example}
    \label{ex:blaschke}
    Let $B(z)$ be a non-constant finite Blaschke product of order $n$, and let $\alpha \in \mathbb{T}$. Then $B(z)$ is an inner function and analytic on $\mathbb{T}$, and $B(\zeta) = \alpha$ has precisely $n$ distinct solutions; denote these by $\eta^\alpha_1, \ldots, \eta^\alpha_n$. Moreover, from properties of finite Blaschke products, its derivative is non-zero on $\mathbb{T}$. By Proposition \ref{one_var pointmass}, the associated Clark measure then satisfies 
    $$ \sigma_\alpha = \sum_{k=1}^n \frac{1}{|B'(\eta^\alpha_k)|}\delta_{\eta^\alpha_k}.
    $$
\end{example}
\begin{example}
\label{exp_ex}
    The function 
    \begin{align*}\phi(z) := \exp \biggl(-\frac{1+z}{1-z}\biggr),  \quad z \in \mathbb{D},\end{align*}
    is inner, and $\phi^*(\zeta)$ exists everywhere on $\mathbb{T}$; this is because
    \begin{align*}
        \bigg| \exp \biggl(-\frac{1+z}{1-z}\biggr) \bigg| = \exp \biggl(\Re \biggl( {-\frac{1+z}{1-z}} \biggr) \biggr) = \exp \biggl(- \frac{1-|z|^2}{|1-z|^2} \biggr), 
    \end{align*}
    from which we can see that $\phi^*( 1) = 0$. Observe that every point $\zeta \neq 1$ on the unit circle solves the equation $\phi^* = \alpha$ for some $\alpha \in \mathbb{T}$, so $\phi^*(\mathbb{T} \setminus \{ 1\} ) = \mathbb{T}$. Moreover, the solutions accumulate in the limit point $\zeta = 1$ for every $\alpha$-value. Since the unimodular level sets are closed by definition, this implies that $1 \in \mathcal{C}_\alpha ( \phi)$ for all $\alpha \in \mathbb{T}$.
    
    Now let $\alpha = 1$ for simplicity. As seen in Example 11.3(ii) in \cite{Intro}, the solutions to $\phi^*( \zeta ) = 1$ are given by 
   \begin{align*}
        \eta_k = \frac{2 \pi k - i}{2 \pi k + i}, \quad k \in \mathbb{Z},
    \end{align*}
    and 
    \begin{align*}
    \frac{1}{|\phi'(\eta_k)|} = \frac{8}{1+4 \pi^2 k^2}.
    \end{align*}
    By Proposition \ref{one_var pointmass}, the Clark measure of $\phi$ associated to $\alpha = 1$ may thus be expressed as
    \begin{align*}
        \sigma_1 = \sum_{k \in \mathbb{Z}} \frac{8}{1+4 \pi^2 k^2} \delta_{\eta_k}.
    \end{align*}
    We will revisit variations of this example in later sections.
\end{example}
By Theorem 4 in \cite{bergqvist}, RP-measures on $\mathbb{T}^d$, $d \geq 2$, cannot be supported on sets of Hausdorff dimension less than one, and in particular, they cannot possess any point masses. One can therefore not hope for an analogous result to Proposition \ref{one_var pointmass} for $d \geq 2$. However, the proposition will still be useful in determining the density of certain Clark measures.

\section{Rational inner functions}
\label{sec:rifs}
There has been significant progress in Clark theory for multivariate rational inner functions, see \cite{clark2} and \cite{clark1}. We already saw a one-variable RIF in Example \ref{ex:blaschke}, where we could describe the Clark measures in a straight-forward manner --- largely because Blaschke products are in fact analytic on $\mathbb{T}$. The main issue when studying RIFs in higher dimensions arises from dealing with potential singularities. However, it turns out that in two variables, the support of any associated Clark measure is actually a finite union of graphs, and that we can explicitly calculate its weights along these. We aim to give an overview of these results in this section. 

We will first need some terminology specific to rational inner functions. We say that a polynomial $p \in \mathbb{C}[z_1,\ldots, z_d]$ is \textit{stable} if it has no zeros in $\mathbb{D}^d$, and that it has \textit{polydegree} $(n_1,\ldots, n_d) \in \mathbb{N}^d$ if $p$ has degree $n_j$ when viewed as a polynomial in $z_j$. By Theorem 5.2.5 in \cite{Rudin}, any rational inner function in $\mathbb{D}^d$ can be written as 
\begin{align*}
    \phi(z) = e^{ia}\prod_{j=1}^d z_j^{k_j} \frac{\tilde{p}(z)}{p(z)}
\end{align*}
where $a \in \mathbb{R}$, $k_1, \ldots, k_d \in \mathbb{N}$, $p$ is a stable polynomial of polydegree $(n_1, \ldots, n_d)$, and
\begin{align*}
\tilde{p}(z) := z_1^{n_1}\cdots z_d^{n_d} \overline{p} \biggl( \frac{1}{\overline{z_1}}, \ldots, \frac{1}{\overline{z_2}} \biggr)
\end{align*}
is its \textit{reflection}. Note that any zero of $p$ will be a zero of $\tilde{p}$ and vice versa, and that $p$ and $\tilde{p}$ have the same polydegree. For simplicity, we will always assume that $\phi(z) = \frac{\tilde{p}}{p}$, where $p$ and $\tilde{p}$ are so called \textit{atoral} --- a concept explored in \cite{atoral} and \cite{local}. In the context of this text, atoral simply means that $p$ and $\tilde{p}$ share no common factors, and that in two dimensions in particular, $p$ and $\tilde{p}$ have finitely many common zeros on $\mathbb{T}^2$ (see Section 2.1, \cite{local}). Hence, a rational inner function $\phi$ in two variables will have at most finitely many singularities on $\mathbb{T}^2$. 

Moreover, we define the polydegree of a rational function $\phi = q/p$ as $(n_1, \ldots, n_d)$, where $p$ and $q$ have no common factors, and $n_j$ is the maximum of the degrees of $p$ and $q$ when viewed as polynomials in variable $z_j$. 
 Thus, the polydegree of $\phi= \tilde{p}/p$ as defined above agrees with the polydegrees of both its numerator and denominator. 

A notable fact about RIFs is that their non-tangential limits exist and are unimodular for every $\zeta \in \mathbb{T}^d$ (Theorem C, \cite{knese}). In \cite{singularities}, the authors prove the following result, which gives us a straight-forward expression for the level sets of RIFs:
\begin{theorem}
    For fixed $\alpha \in \mathbb{T}$, let 
    \begin{align*}
    \mathcal{L}_{\alpha}( \phi) := \{ \zeta \in \mathbb{T}^d: \tilde{p}( \zeta) - \alpha p(\zeta) = 0 \}. 
    \end{align*}
    Then $\mathcal{C}_\alpha ( \phi ) = \mathcal{L}_{\alpha}( \phi ) .$
\end{theorem}
\begin{proof}
    See Theorem 2.6 in \cite{singularities}.
\end{proof}
Note that for any zero of $p$, the equation $\tilde{p}- \alpha p = 0$ is trivially satisfied. This implies that all singularities of $\phi$ on $\mathbb{T}^d$ are contained in $\mathcal{C}_\alpha( \phi)$. Moreover, observe that $\{ \phi^* = \alpha \}$ is in general not a closed set. However, by the theorem above, we may characterize $\mathcal{C}_\alpha ( \phi)$ as the zeros of a polynomial when $\phi$ is a RIF. 

When $d=2$, one can find an even nicer characterization of the unimodular level sets:
\begin{lemma}
    Let $\phi= \frac{\tilde{p}}{p}$ be a RIF of bidegree $(m,n)$, and fix $\alpha \in \mathbb{T}$. For any choice of $\tau_0 \in \mathbb{T}$, there exists a finite number of functions $g^{\alpha}_1, \ldots, g^{\alpha}_n$ defined on $\mathbb{T}$ and analytic on $\mathbb{T} \setminus \{ \tau_0 \} $ such that $\mathcal{C}_\alpha( \phi)$ can be written as a union of curves
    \begin{align*}
        \{ ( \zeta, g^{\alpha}_j ( \zeta)): \zeta \in \mathbb{T}\}, \quad j= 1, \ldots, n,
    \end{align*}
    potentially together with a finite number of vertical lines $\zeta_1 = \tau_1, \ldots, \zeta_1= \tau_k$, where each $\tau_j \in \mathbb{T}$.
    \label{parameter}
\end{lemma}

\begin{proof}
    See Lemma 2.5 in \cite{clark2}. 
\end{proof}

The proof is quite technical and will only be outlined here. Specifically, the authors fix a point $\tau \in \mathbb{T}$ and construct a parameterization of $\mathcal{C}_\alpha(\phi) \cap (I_\tau \times \mathbb{T})$ for a small interval $I_\tau \ni \tau$ in $\mathbb{T}$. When $\tau$ is not the $z_1$-coordinate of a singularity of $\phi$, one can use properties of RIFs to show that $\phi(\tau,\cdot)$ satisfies the conditions of the Implicit Function Theorem. Hence the solutions to $\phi^* = \alpha$ can be parameterized by smooth curves on the strip $I_\tau \times \mathbb{T}$. 

The main issue arises from the fact that the curves $g_j^\alpha$ might intersect at singularities of $\phi$, in which case analyticity is not obvious. One must thus ensure that we can in some sense “pull apart” any crossed curves in $\mathcal{C}_\alpha(\phi)$ and prove that they are each analytic when viewed separately. However, it is shown in  \cite{local} that near each singularity of $\phi$, the level sets actually do consist of smooth curves. Hence, even in the case where $\tau$ is the $z_1$-coordinate of a singularity, one obtains a smooth parameterization of $I_\tau \times \mathbb{T}$. In the last step of the proof, the authors glue together the local parameterizations, which yields the final result. 

The analysis of Clark measures of RIFs must now be divided into two cases; when the unimodular constant $\alpha $ is \textit{generic} versus \textit{exceptional} as defined below.
\begin{definition}
    We say that $\alpha \in \mathbb{T}$ is an exceptional value if $\phi(\tau, \zeta_2) \equiv \alpha$ or  $\phi(\zeta_1, \tau) \equiv \alpha$ for some $\tau \in \mathbb{T}$. If $\alpha$ is not exceptional, we say that it is generic.
    \label{def_gen_exc}
\end{definition}
The different cases stem from the characterization of $\mathcal{C}_\alpha ( \phi) $ in Lemma \ref{parameter}; if $\alpha$ is an exceptional value, by the definition above, the level sets will contain lines of the form $ \{ \zeta_1 = \tau \}$ or $ \{ \zeta_2 = \tau \}$. If $\alpha$ is generic, $\mathcal{C}_\alpha ( \phi)$ can be fully described by the graphs of the functions $g^{\alpha}_1, \ldots, g^{\alpha}_n$. 

\begin{theorem}
    Let $\phi = \frac{\tilde{p}}{p}$ be a RIF of bidegree $(m,n)$ and $\alpha \in \mathbb{T}$ a generic value for $\phi$. Then the associated Clark measure $\sigma_\alpha$ satisfies 
    \begin{align*}
        \int_{\mathbb{T}^2} f( \xi) d \sigma_\alpha ( \xi) = \sum_{j=1}^n \int_\mathbb{T} f( \zeta, g^{\alpha}_j ( \zeta)) \frac{dm(\zeta)}{|\frac{\partial \phi}{\partial z_2} ( \zeta, g^{\alpha}_j ( \zeta) ) |}
    \end{align*}
    for all $f \in C ( \mathbb{T}^2) $, where $g^{\alpha}_1, \ldots, g^{\alpha}_n$ are the parameterizing functions from Lemma \ref{parameter}.
    \label{generic}
\end{theorem}

\begin{proof}
    See Theorem 3.3 in \cite{clark2}. 
\end{proof}

\begin{theorem}
    Let $\phi = \frac{\tilde{p}}{p}$ be a RIF of bidegree $(m,n)$ and $\alpha \in \mathbb{T}$ an exceptional value for $\phi$. Then, for $f \in C(\mathbb{T}^2)$, the associated Clark measure $\sigma_\alpha$ satisfies 
    \begin{align*}
        \int_{\mathbb{T}^2} f( \xi) d \sigma_\alpha ( \xi ) = \sum_{j=1}^n \int_\mathbb{T} f( \zeta, g^{\alpha}_j ( \zeta)) \frac{dm(\zeta)}{|\frac{\partial \phi}{\partial z_2} ( \zeta, g^{\alpha}_j ( \zeta) ) |} + \sum_{k=1}^\ell  c^\alpha_k \int_\mathbb{T}f( \tau_k, \zeta) dm(\zeta),
    \end{align*}
    where $g_1^\alpha, \ldots, g_n^\alpha$ are the parameterizing functions and $\zeta_1 = \tau_1, \ldots, \zeta_1 = \tau_\ell$ the vertical lines in $\mathcal{C}_\alpha ( \phi)$ from Lemma \ref{parameter}, and $c_k^\alpha := 1/|\frac{\partial \phi}{\partial z_1}(\tau_k, z_2)| > 0$ are constants.
    \label{exceptional}
\end{theorem}

\begin{proof}
    See Theorem 3.8 in \cite{clark2}.
\end{proof}
Note that in both the generic and exceptional case, the weights of the Clark measures along level curve components are generally given by one-variable functions. 
The fine structure of RIF weights is thoroughly analyzed in \cite{clark2} --- we will only briefly touch upon this here. For $\phi = \tilde{p} / p$, let $W_j^\alpha(\zeta) :=  |\frac{\partial \phi}{\partial z_2} ( \zeta, g^{\alpha}_j ( \zeta) ) |^{-1}$  denote the weights from Theorem \ref{generic} and Theorem \ref{exceptional}. Then, by Lemma 5.1 in \cite{clark2}, these functions are in $L^1(\mathbb{T})$ and may be expressed as
\begin{align*}
    W_j^\alpha(\zeta) = \frac{|p(\zeta, g_j^\alpha ( \zeta)|}{| \frac{\partial \tilde{p}}{\partial z_2}(\zeta, g_j^\alpha(\zeta)) - \alpha \frac{\partial p}{\partial z_2} (\zeta, g_j^\alpha(\zeta)) |}.
\end{align*}
As we established earlier, if $(\tau, \gamma) \in \mathbb{T}^2$ is a singularity of $\phi$, then $(\tau, \gamma) \in \mathcal{C}_\alpha(\phi)$ for every $\alpha \in \mathbb{T}$. If $(\tau, \gamma) = (\tau, g_j^\alpha(\tau))$ for some level curve $g_j^\alpha$, then $p(\tau, g_j^\alpha(\tau)) = 0$ and one might expect $W_j^\alpha$ to be zero at this point --- at least if there is no cancellation from the denominator. However, it could be that there are curve components $g_j^\alpha$ in $\mathcal{C}_\alpha(\phi)$ which do not satisfy $g_j^\alpha(\tau) = \gamma$. In \cite{clark2}, the authors introduce the notion of \textit{contact order} and prove the following statement:
\begin{center}
\textit{For all but finitely many $\alpha$, if a branch $(\zeta, g_j^\alpha(\zeta))$ of $\mathcal{C}_\alpha(\phi)$ passes through the singularity $(\tau, \gamma)$, the corresponding weight function $W_j^\alpha$ has order of vanishing at $\tau$ that corresponds to the contact order of the corresponding branch of $\mathcal{Z}(\tilde{p})$ at $(\tau, \gamma)$.}
\end{center}
Here, $\mathcal{Z}(\tilde{p})$ denotes the zero set of the polynomial $\tilde{p}$. The precise statement can be found in Theorem 5.6 in \cite{clark2}; the gist is that there are constants $c,C$ such that one can bound
$$ 0 < c \leq \frac{W_j^\alpha(\zeta)}{|\zeta-\tau|^{K_j}} \leq C
$$
for all $\zeta$ in a neighborhood of $\tau$, where $K_j$ is the contact order of $\phi$ at $(\tau, \gamma) = (\tau, g_j^\alpha(\tau))$ associated with the branch $g_j^\alpha$. Consequently, under these conditions, $W_j^\alpha$ is a bounded function. 

The case of RIFs of bidegree $(n,1)$ specifically has been studied in great detail in \cite{clark1}. For these functions, we obtain a more explicit version of Theorem \ref{exceptional}. If $\phi = \tilde{p}/p$ has bidegree $(n,1)$, we may write
    \begin{align*}
        p(z) = p_1(z_1) + z_2 p_2(z_1)
    \quad \text{and} \quad
    \tilde{p}(z) = z_2 \tilde{p}_1(z_1)+ \tilde{p}_2(z_1)
    \end{align*}
    for reflections $\tilde{p}_i = z_1^n \overline{p}_i(1/\overline{z}_i)$. In this case, solving $\phi^* = \alpha$ for $z_2$ yields $z_2 = \frac{1}{B_\alpha(z_1)}$, where 
    \begin{align*} 
    B_\alpha(z) := \frac{\tilde{p}_1(z) - \alpha p_2(z)}{\alpha p_1(z) - \tilde{p}_2(z)}.
    \end{align*}
    Moreover, define
    \begin{align*} 
    W_\alpha(\zeta) := \frac{|p_1(\zeta)|^2-|p_2(\zeta)|^2}{|\tilde{p}_1(\zeta) - \alpha p_2(\zeta)|^2}. 
    \end{align*}
    Then, by Theorem 1.2 in \cite{clark1}, we have
    \begin{align*}
        \int_{\mathbb{T}^2} f( \xi) d \sigma_\alpha ( \xi ) = \int_\mathbb{T} f( \zeta, \overline{B_\alpha ( \zeta)}) W_\alpha( \zeta) dm(\zeta) + \sum_{k=1}^\ell  c^\alpha_k \int_\mathbb{T}f( \tau_k, \zeta) dm(\zeta)
    \end{align*}
    with $c_k^\alpha = 1/|\frac{\partial \phi }{\partial z_1}(\tau_k, z_2)|$ is non-zero if and only if $\alpha$ is an exceptional value. It is worth noting that for any RIF $\phi$ of bidegree $(n,1)$, a value $\alpha \in \mathbb{T}$ is exceptional if and only if it is the non-tangential value of $\phi$ at some singularity (see Section 3 of \cite{clark1}).

\begin{example}
\label{ex:rif}
    For an explicit example, we use Example 5.2 from \cite{clark1}: let $\phi = \frac{\tilde{p}}{p}$ for 
    \begin{align*}
        p(z) = 4-z_2-3z_1-z_1z_2+z_1^2 \quad \text{and} \quad \tilde{p}(z) = 4z_1^2z_2 - z_1^2- 3z_1 z_2- z_1 + z_2.
    \end{align*}
    Observe that $\phi$ has only one singularity, which occurs at $(1,1)$. For each $\alpha \in \mathbb{T}$, the formulas above yield 
    \begin{align*}
        B_\alpha (z ) = \frac{4z_1^2 - 3z_1 + 1 +\alpha + \alpha z_1}{4 \alpha - 3z_1 \alpha + z_1^2 \alpha +z_1^2 + z_1}
    \end{align*}
    and
    $$ W_\alpha ( \zeta) = \frac{4|\zeta-1|^4}{|4 \zeta^2 - 3 \zeta + 1 + \alpha + \alpha \zeta|^2}.
    $$
    We see that $\alpha = -1$ is an exceptional value, as $\phi = -1$ is solved by $(1, z_2)$ as well as $\big(z_1, \frac{1}{B_{-1}(z_1)}\big) = (z_1, 1/z_1)$. Since $\phi$ only has one singularity, this point gives rise to the only exceptional value and $\phi^*(1,1) = -1$. Hence, for $\alpha \neq -1$, we have
    \begin{align*}
        \int_{\mathbb{T}^2} f( \xi) d \sigma_\alpha ( \xi ) = \int_\mathbb{T} f( \zeta, \overline{B_\alpha ( \zeta)})   \frac{4|\zeta-1|^4}{|4 \zeta^2 - 3 \zeta + 1 + \alpha + \alpha \zeta|^2} dm(\zeta).
    \end{align*}
    Moreover, we see that $W_{-1} ( \zeta ) = \frac{1}{4}|\zeta-1|^2$ and $\frac{\partial \phi }{\partial z_1}(1,z_2) = -2$, which yields 
    \begin{align*}
        \int_{\mathbb{T}^2} f( \xi) d \sigma_{-1} ( \xi ) =   \frac{1}{4} \int_\mathbb{T} f( \zeta, \overline{\zeta})|\zeta-1|^2 dm(\zeta) + \frac{1}{2}\int_{\mathbb{T}} f( 1, \zeta) dm (\zeta)
    \end{align*}
    for $\alpha = -1$. In Figure \ref{fig:rif}, we have plotted the level curves corresponding to different $\alpha$-values.
    
\begin{figure}
    \centering
    \includegraphics[width=7cm]{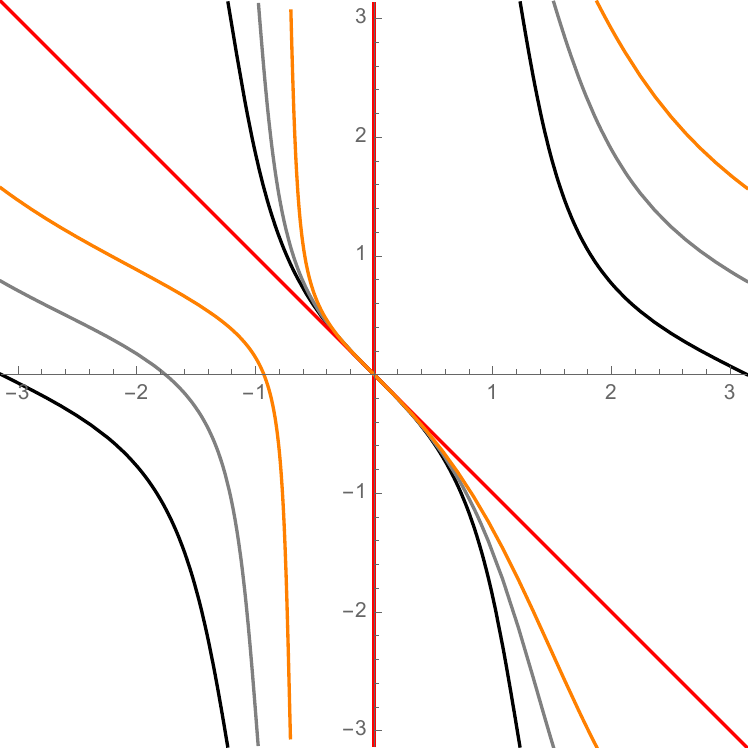}
    \caption{Level curves in Example \ref{ex:rif} for $\alpha = 1$ (black), $\alpha = e^{i\pi/4}$ (gray), $\alpha = e^{i \pi/2}$ (orange), and exceptional value $\alpha = - 1$ (red). }
    \label{fig:rif}
\end{figure}
\end{example}


\section{Multiplicative embeddings}
\label{sec:mult}

Given an inner function $\phi$ in one complex variable, we define the multiplicative embedding
$$ \Phi(z) = \Phi(z_1, z_2) := \phi(z_1 z_2), \quad z \in \mathbb{D}^2.
$$
The function defined by $(z_1, z_2) \mapsto z_1 z_2$ maps $\mathbb{D}^2$ to $\mathbb{D}$, and so $\phi$ being an inner function implies that $\Phi$ is inner as well. In the following proposition, we characterize the support set of $\Phi$ with the help of the original function $\phi$. 
\begin{proposition}
    Let $\phi(z)$ be an inner function in one variable, and $\alpha \in \mathbb{T}$. Define $\Phi(z_1,z_2) := \phi(z_1 z_2) $. Then 
    \begin{align*}
        \mathcal{C}_\alpha ( \Phi) = \bigcup_{\zeta \in \mathcal{C}_\alpha( \phi ) } \{ (z, \zeta \overline{z} ): z \in \mathbb{T} \}. 
    \end{align*}
    \label{cformula}
\end{proposition}

\begin{proof}
    First, for ease of notation, define 
    $$\mathcal{C}_\alpha ' (f) := \Bigl\{ \zeta \in \mathbb{T}^d: \lim_{r \to 1^-} f(r \zeta) = \alpha \Bigl\} $$
    for any inner function $f$, so that $\mathcal{C}_\alpha (f) = \text{Clos} (\mathcal{C}_\alpha ' (f)) $. 
    
    Let $\zeta \in \mathcal{C}_\alpha ' ( \phi )$. Then we know that $\lim_{r \to 1^-} \phi(r \zeta ) = \alpha. 
    $
    For every $z \in \mathbb{T}$, 
    \begin{align*}
        \lim_{r \to 1^-} \Phi(r (z, \zeta \overline{z})) = \lim_{r \to 1^-} \phi (r^2 \zeta z \overline{z} ) = \lim_{r \to 1^-} \phi (r \zeta) = \alpha,
    \end{align*}
    implying that $(z, \zeta \overline{z} ) \in \mathcal{C}_\alpha ( \Phi )$. Thus 
    \begin{align*}
        \bigcup_{\zeta \in \mathcal{C}_\alpha'( \phi ) } \{ (z, \zeta \overline{z} ): z \in \mathbb{T} \} \subset \mathcal{C}_\alpha ( \Phi ).
    \end{align*}
    To extend this to a union over $\mathcal{C}_\alpha ( \phi)$, let $\zeta \in \mathcal{C}_\alpha ( \phi ) $. Then there exists some sequence $ ( \zeta_n)_{n\geq1} $ in $\mathcal{C}_\alpha'( \phi )$ that converges to $\zeta $ as $n$ tends to infinity. This also implies that for any $z \in \mathbb{T}$, $(z, \zeta_n \overline{z}) \to (z, \zeta \overline{z}) \in \mathcal{C}_\alpha (\Phi) $ as $n \to \infty$. Hence, 
     \begin{align*}
        \bigcup_{\zeta \in \mathcal{C}_\alpha( \phi ) } \{ (z, \zeta \overline{z} ): z \in \mathbb{T} \} \subset \mathcal{C}_\alpha ( \Phi ).
    \end{align*}
    
    Conversely, let $(z_1, z_2) \in \mathcal{C}_\alpha '( \Phi )$, so 
    \begin{align*}
        \lim_{r \to 1^-} \Phi(r (z_1, z_2)) = \lim_{r \to 1^-} \phi (r^2 z_1z_2 ) = \alpha.
    \end{align*}
    Then $\zeta := z_1 z_2 \in  \mathcal{C}_\alpha ( \phi )$. Since $z_1, z_2 \in \mathbb{T}$, we may write 
    \begin{align*}
        z_2 = \frac{\zeta}{z_1} = \zeta \overline{z_1}, 
    \end{align*}
    so $(z_1, z_2) = (z_1, \zeta \overline{z_1} ) \in \{ (z, \zeta \overline{z}) : z \in \mathbb{T} \}$. Hence, 
    \begin{align*}
         \mathcal{C}_\alpha' ( \Phi ) \subset \bigcup_{\zeta \in \mathcal{C}_\alpha( \phi ) } \{ (z, \zeta \overline{z} ): z \in \mathbb{T} \}.
    \end{align*}
    Now let $(z_1, z_2) \in \mathcal{C}_\alpha ( \Phi) $. Then there is some sequence of $(z_{1,n}, z_{2,n} )$ in $\mathcal{C}_\alpha ' ( \Phi) $ converging to $(z_1, z_2) $ as $n \to \infty$. But this implies that $z_{1,n} z_{2,n} \to z_1 z_2 \in \mathcal{C}_\alpha( \phi ) $, and the same argument as above then yields 
    \begin{align*}
        \mathcal{C}_\alpha ( \Phi ) \subset \bigcup_{\zeta \in \mathcal{C}_\alpha( \phi ) } \{ (z, \zeta \overline{z} ): z \in \mathbb{T} \}.
    \end{align*} 
    \end{proof}
As in the RIF case, the unimodular level sets of this class of functions may be expressed as unions of curves. However, as opposed to in Lemma \ref{parameter}, the unions need not be finite --- or even countable --- here. 
\begin{remark}
    Observe that by Lemma 2.2 in \cite{clark2}, any positive, pluriharmonic, $m_d$-singular probability measure defines the Clark measure of some inner function. Hence, there exist Clark measures with significantly more intricate supports than what we have seen so far. 
    \label{skeptic}
\end{remark}
A natural next step is to investigate whether we can characterize the density of a given Clark measure $\tau_\alpha$ on the antidiagonals in $\mathcal{C}_\alpha ( \Phi )$. We do this in the next result and its subsequent corollary: 
\begin{theorem}
    \label{zw-thm}
    Let $\phi(z)$ be an inner function in one variable, with Clark measure $\sigma_\alpha$ for some unimodular constant $\alpha$. Let $\tau_\alpha$ be the corresponding Clark measure of $\Phi(z_1, z_2) := \phi(z_1 z_2)$. Then, for any function $f \in C(\mathbb{T}^2)$, 
    \begin{align*}
        \int_{\mathbb{T}^2} f( \xi ) d\tau_\alpha ( \xi ) = \int_\mathbb{T} \bigg( \int_\mathbb{T} f(\zeta, x \overline{\zeta}) d \sigma_\alpha(x) \bigg) dm(\zeta). 
    \end{align*}
\end{theorem}
\begin{proof}
    We first prove this in the case when $f$ is the product of one-variable Poisson kernels. Fixing $z_2 \in \mathbb{D}$, let
    \begin{align*}
        u_{z_2}( z_1) := \frac{1-|\phi(z_1 z_2)|^2}{|\alpha-\phi(z_1 z_2)|^2} = \int_{\mathbb{T}^2} P_{z_1}( \xi_1) P_{z_2}( \xi_2) d \tau_\alpha(\xi), \quad z_1 \in \mathbb{D}. 
    \end{align*}
    As the middle expression is pluriharmonic, $u_{z_2}$ must be harmonic on $\mathbb{D}$. Since $z_1 z_2 \in \mathbb{D}$ for any $z_1 \in \overline{\mathbb{D}}$, and $\phi$ is analytic (and hence continuous) on $\mathbb{D}$, we see that $\Phi(z_1, z_2) = \phi(z_1 z_2)$ as a function of $z_1$ is continuous on $\overline{\mathbb{D}}$. Moreover, by the maximum principle, $|\phi| < 1$ on the unit disc, which implies that the denominator will always be non-zero. We conclude that $u_{z_2}$ is continuous in $\overline{\mathbb{D}}$, and we may thus apply the Poisson integral formula: 
    \begin{align}
        u_{z_2}( z_1 ) =  \int_\mathbb{T} \frac{1-|\phi(\zeta z_2)|^2}{|\alpha-\phi(\zeta z_2)|^2} P_{z_1}( \zeta) dm(\zeta).
        \label{pif}
    \end{align}
    Moreover, for $\zeta \in \mathbb{T}$, we see that
    \begin{align*}
    \int_{\mathbb{T}} P_{z}( \zeta, x \overline{\zeta} ) d \sigma_\alpha(x) &= \int_{\mathbb{T}} P_{z_1}( \zeta) P_{z_2} (x \overline{\zeta} ) d \sigma_\alpha(x)\\
    &= P_{z_1}( \zeta) \int_{\mathbb{T}} \frac{1-|z_2|^2}{|x \overline{\zeta}- z_2|^2} d\sigma_\alpha(x) \\
    &= P_{z_1}( \zeta) \int_{\mathbb{T}} \frac{1-|\zeta z_2|^2}{|x- \zeta z_2|^2}  d\sigma_\alpha(x) \\
    &= P_{z_1}( \zeta) \int_{\mathbb{T}} P_{\zeta z_2}(x) d\sigma_\alpha(x) \\
    &= P_{z_1} ( \zeta) \frac{1-|\phi( \zeta z_2)|^2 }{| \alpha - \phi(\zeta z_2)|^2},
    \end{align*}
    where we use the definition of the Clark measure $\sigma_\alpha$ in the last step. By integrating the above and applying \eqref{pif}, we get 
    \begin{align*}
        \int_\mathbb{T} \bigg( \int_{\mathbb{T}} P_{z} ( \zeta, x \overline{\zeta} ) d \sigma_\alpha(x)  \bigg) dm(\zeta) &= \int_\mathbb{T} \frac{1-|z_1|^2}{|\zeta-z_1|^2} \frac{1-|\phi(\zeta z_2)|^2}{|\alpha-\phi(\zeta z_2)|^2} dm(\zeta) \\
        &= \frac{1-|\phi(z_1 z_2)|^2}{|\alpha-\phi(z_1 z_2)|^2} \\
        &= \int_{\mathbb{T}^2} P_{z_1}( \xi_1) P_{z_2}( \xi_2) d \tau_\alpha(\xi).
    \end{align*}
    Now apply Lemma \ref{density} to obtain the final result. \end{proof}

\begin{remark}
It is a priori not obvious that $f( \zeta, x \overline{\zeta})$ is integrable with respect to $\sigma_\alpha$. Integrability is ensured by the fact that $f( \zeta, x \overline{\zeta})$ is continuous on $\mathbb{T}$, as it is composed by two functions $f$ and $g_x(z) := ( z, x \overline{z})$ which are continuous there. Since $\sigma_\alpha$ is a finite, positive Borel measure on a compact space, all continuous functions on said space are integrable with respect to $\sigma_\alpha$. 
\end{remark}

In particular, when the Clark measures associated to $\phi$ are discrete, one gets the following result:
\begin{corollary}  Let $\phi: \mathbb{D} \to \mathbb{C}$ be an inner function with Clark measure $\sigma_\alpha$ for some unimodular constant $\alpha$, and let $\tau_\alpha$ be the Clark measure of $\Phi(z_1, z_2) := \phi(z_1 z_2)$. If $\sigma_\alpha$ is supported on a countable collection of points $\{ \eta_k \}_{k \geq 1} \subset \mathcal{C}_\alpha ( \phi ) $, then
\begin{align*}
    \int_{\mathbb{T}^2} f( \xi ) d\tau_\alpha ( \xi ) =  \sum_{k\geq1} \int_\mathbb{T} f(\zeta, \eta_k \overline{\zeta})  \frac{ dm(\zeta)}{|\phi'(\eta_k)|} 
\end{align*}
for all $f \in C(\mathbb{T}^2)$. 
\label{zw_formula}
\end{corollary}

\begin{proof}
    By Proposition \ref{one_var pointmass}, $\sigma_\alpha$ having a point mass at some $\eta_k$ implies that $\sigma_\alpha( \{ \eta_k \}) =1 / |\phi'(\eta_k)|$. Then, following the steps in the proof of Theorem \ref{zw-thm}, 
     \begin{align*}
    \int_{\mathbb{T}} P_{z}( \zeta, x \overline{\zeta} ) d \sigma_\alpha(x) 
    &= P_{z_1}( \zeta) \int_{\mathbb{T}}  \Bigg( \sum_{k\geq1} \frac{1}{|\phi'(\eta_k)|}
    P_{\zeta z_2}(x) \Bigg) d \delta_{\eta_k}(x). \end{align*}
    This then reduces to
    \begin{align*}
    P_{z_1}( \zeta) \sum_{k\geq1} \frac{1}{|\phi'(\eta_k)|} 
    P_{\zeta z_2}(\eta_k) =  \sum_{k\geq1} \frac{1}{|\phi'(\eta_k)|}  P_{z_1}( \zeta) P_{z_2}( \eta_k \overline{\zeta}).
    \end{align*}
    Hence,
    \begin{align*}
        \int_{\mathbb{T}} P_{z}( \zeta, x \overline{\zeta} ) d \sigma_\alpha(x) = \sum_{k\geq1}\frac{1}{|\phi'(\eta_k)|} P_{z}( \zeta, \eta_k \overline{\zeta}).
    \end{align*}
    Integrating over this and applying Theorem \ref{zw-thm} then shows that 
    \begin{align*}
        \int_{\mathbb{T}^2} P_{z}( \xi) d \tau_\alpha(\xi) =  \int_\mathbb{T} \Bigg( \sum_{k\geq 1} \frac{1}{|\phi'(\eta_k)|}  P_{z}( \zeta, \eta_k \overline{\zeta}) \Bigg) d m (\zeta) 
        =  \sum_{k\geq 1} 
    \int_\mathbb{T}  P_{z}( \zeta, \eta_k \overline{\zeta}) \frac{d m (\zeta)}{|\phi'(\eta_k)|},
    \end{align*}
    where we have used positivity of the summands in the last step. The result now follows from Lemma \ref{density}. \end{proof}



It is interesting to compare the above result to the corresponding theorems, Theorem \ref{generic} and Theorem \ref{exceptional}, for rational inner functions. In the RIF case, we saw that the weights of Clark measures along the curves in the unimodular level sets were one-variable functions. Corollary \ref{zw_formula} shows that for the multiplicative embeddings, the weights are simpler than their RIF counterparts --- they are constant along each curve in the level sets. This implies that given any univariate inner function $\phi$, regardless of its complextiy, the associated Clark measures of $\phi(z_1 z_2)$ will always be very “well-behaved”, in the sense that they are supported on straight lines and --- when the Clark measures of $\phi$ are discrete --- have constant density along each such line.

\begin{example}
Recall the function 
    \begin{align*}\phi(z) := \exp \biggl(-\frac{1+z}{1-z}\biggr),  \quad z \in \mathbb{D},\end{align*}
    from Example \ref{exp_ex}. We saw there that the solutions to $\phi^*( \zeta ) = 1$ are given by 
   \begin{align*}
        \eta_k = \frac{2 \pi k - i}{2 \pi k + i}, \quad k \in \mathbb{Z},
    \end{align*}
    and 
    \begin{align*}
    \frac{1}{|\phi'(\eta_k)|} = \frac{8}{1+4 \pi^2 k^2}.
    \end{align*}
    Now consider $\Phi(z_1, z_2) := \phi(z_1 z_2)$, which appears in e.g. Example 13.1 in \cite{der_of_rifs}. Applying Corollary \ref{zw_formula} for the Clark measure $\tau_1$ of $\Phi$ then results in 
    \begin{align*} 
   \int_{\mathbb{T}^2} f( \xi) d \tau_1 ( \xi ) =  \sum_{k \in \mathbb{Z}}\frac{8}{1+4 \pi^2 k^2} \int_\mathbb{T}   f( \zeta, \eta_k \overline{\zeta } )dm(\zeta)
\end{align*}
for $f \in C(\mathbb{T}^2)$. This marks our first example of a non-rational bivariate function, for which we can explicitly characterize the Clark measures. Moreover, this is our first example of an inner function whose unimodular level sets consist of infinitely many curves, as opposed to the RIF case. 
\label{exp_ex2}
\end{example}

We now extend this theory to $d$ variables. For an inner function $\phi$ in one variable, define the multiplicative embedding 
$$ \Phi(z) := \phi(z_1 z_2 \cdots z_d), \quad z \in \mathbb{D}^d.
$$
By the same argument as for two variables, this is an inner function. In the next result, we prove a $d$-dimensional version of Theorem \ref{zw-thm}:

\begin{theorem}
    Let $\phi(z): \mathbb{D} \to \mathbb{C}$ be an inner function with Clark measure $\sigma_\alpha$ for some unimodular constant $\alpha$, and let $\tau_\alpha$ be the Clark measure of $\Phi(z_1, \ldots, z_d) := \phi(z_1 z_2 \cdots z_d)$. Then 
    \begin{align*}
        \int_{\mathbb{T}^d} f(\xi) d\tau_\alpha(\xi) =  \int_{\mathbb{T}} \int_{\mathbb{T}}\cdots \int_{\mathbb{T}} f(\zeta_1, \ldots, \zeta_{d-1}, x\overline{\zeta_1 \cdots \zeta_{d-1}} ) d\sigma_\alpha(x) dm(\zeta_{d-1}) \cdots dm(\zeta_1).
    \end{align*}
    \label{d-dim formula}
\end{theorem}

\begin{proof}
    We prove the result by induction, where Theorem \ref{zw-thm} marks our base case. As usual, we prove the formula for Poisson kernels first.

    We begin by introducing some notation: for $n < m$, let
    $$ \bold{z}_n^m := z_n \cdots z_m
    $$
    for $z_j \in \mathbb{D}, n \leq j \leq m$. Note that $\Phi(z) = \phi(\bold{z}_1^d)$ per definition.
    
    Suppose the formula holds for $\phi(z_1 \cdots z_{d-1})= \phi(\bold{z}_1^{d-1})$, in which case
    \begin{align*}
        \frac{1-|\phi(\bold{z}_1^{d-1})|^2}{|\alpha - \phi(\bold{z}_1^{d-1})|^2} = \int_{\mathbb{T}} \int_{\mathbb{T}} \cdots \int_{\mathbb{T}} P_{z_1}(\zeta_1) \cdots P_{z_{d-1}}(x\overline{\zeta_1 \zeta_2 \cdots \zeta_{d-2}} ) d\sigma_\alpha(x) dm(\zeta_{d-2}) \cdots dm(\zeta_1).
    \end{align*}
    We now want to show the result for $d$ variables. Fix $z_2, \ldots, z_d$ and define the one-variable function
$$u(z_1) :=  \frac{1-|\phi(\bold{z}_1^{d})|^2}{|\alpha - \phi(\bold{z}_1^{d})|^2} = \frac{1-|\phi(z_1 \cdot \bold{z}_2^{d})|^2}{|\alpha - \phi(z_1 \cdot \bold{z}_2^{d})|^2}, \quad z_1 \in \mathbb{D}.
$$
By the same argument as in the proof of Theorem \ref{zw-thm}, this function is harmonic on the unit disc and continuous on its closure. Hence, we may apply the Poisson integral formula: 
\begin{align} \frac{1-|\phi(\bold{z}_1^{d})|^2}{|\alpha - \phi(\bold{z}_1^{d})|^2} = \int_{\mathbb{T}^d} P_{z_1}(\zeta_1) \frac{1-|\phi(\bold{z}_2^d \cdot \zeta_1)|^2}{|\alpha - \phi(\bold{z}_2^d \cdot \zeta_1)|^2} 
\label{poisson eq d-dim}
\end{align}
For fixed $\zeta_1 \in \mathbb{T}$, define $\phi_{\zeta_1}(\bold{z}_2^d) := \phi(\bold{z}_2^d \cdot \zeta_1)$. Then, by our induction assumption, it holds that
\begin{align*}
    \frac{1-|\phi_{\zeta_1}(\bold{z}_2^d)|^2}{|\alpha - \phi_{\zeta_1}(\bold{z}_2^d)|^2} &= \int_{\mathbb{T}} \int_{\mathbb{T}} \cdots \int_{\mathbb{T}} P_{z_2}(\zeta_2) \cdots P_{\zeta_1 z_{d-1}}(x\overline{\zeta_2 \cdots \zeta_{d-1}} ) d\sigma_\alpha(x) dm(\zeta_{d-1}) \cdots dm(\zeta_2) \\
    &= \int_{\mathbb{T}} \int_{\mathbb{T}} \cdots \int_{\mathbb{T}} P_{z_2}(\zeta_2) \cdots P_{z_{d-1}}(x\overline{\zeta_1 \zeta_2 \cdots \zeta_{d-1}} ) d\sigma_\alpha(x) dm(\zeta_{d-1}) \cdots dm(\zeta_2).
\end{align*}
Finally, by \eqref{poisson eq d-dim}, we arrive at 
\begin{align*}
    \frac{1-|\phi(\bold{z}_1^{d})|^2}{|\alpha - \phi(\bold{z}_1^{d})|^2} = 
    \int_{\mathbb{T}} \int_{\mathbb{T}} \cdots \int_{\mathbb{T}} P_{z_1}(\zeta_1) \cdots P_{z_{d-1}}(x\overline{\zeta_1 \cdots \zeta_{d-1}}) d\sigma_\alpha(x) dm(\zeta_{d-1}) \cdots dm(\zeta_1),
\end{align*}
as desired. Application of Lemma \ref{density} yields the final result. 
\end{proof}
Similarly to in the two-variable case, this gives us a sense of the geometry of $\supp\{ \tau_\alpha \}$. For example, for $d=3$ and $x = e^{i \nu} \in \mathbb{T}$, the set
$$ \{ (\zeta_1, \zeta_2, x \overline{\zeta_1 \zeta_2}): \zeta_1, \zeta_2 \in \mathbb{T} \} = \{ (e^{is}, e^{it}, e^{i(\nu -s-t)} ): -\pi \leq s,t \leq \pi \}$$
has logarithmic coordinates $(s, t, \nu -s-t)$, which defines a plane in $\mathbb{T}^3$.

As in the case of $d=2$, we get the following consequence when $\sigma_\alpha$ is discrete: 
\begin{corollary}
    \label{d-dim cor}
    Let $\phi(z): \mathbb{D} \to \mathbb{C}$ be an inner function with Clark measure $\sigma_\alpha$ for some unimodular constant $\alpha$, and let $\tau_\alpha$ be the Clark measure of $\Phi(z_1, \ldots, z_d) := \phi(z_1 z_2 \cdots z_d)$. If $\sigma_\alpha$ is supported on a countable collection of points $\{\eta_k \}_{k \geq 1} \subset \mathcal{C}_\alpha (\phi)$, then 
    \begin{align*}
        \int_{\mathbb{T}^d} f(\xi) d\tau_\alpha(\xi) = \sum_{k \geq 1} \int_{\mathbb{T}} \cdots \int_{\mathbb{T}} f(\zeta_1, \ldots, \zeta_{d-1}, \overline{\zeta_1 \cdots \zeta_{d-1}} \eta_k) dm(\zeta_{d-1})  \cdots  \frac{dm(\zeta_1)}{|\phi'(\eta_k)|}.
    \end{align*}
\end{corollary}

\section{Product functions}
\label{sec:prod}
Given one-variable inner functions $\phi$ and $\psi$, define the product function
\begin{align*}
    \Psi(z) := \phi(z_1) \psi(z_2), \quad z \in \mathbb{D}^2.
\end{align*}
Then $\Psi$ is an inner function in $\mathbb{D}^2$. The analysis of the Clark measures of $\Psi$ is not as straight-forward as for the multiplicative embeddings. A key argument in the proofs of Theorem \ref{generic}, Theorem \ref{exceptional} and \ref{zw-thm} is the Poisson integral formula. To use this for $\Psi(z_1, z_2)$, we require that for fixed $z_2 \in \mathbb{D}$, the function
\begin{align*}
    u_{z_2}(z_1) := \frac{1- |\phi(z_1)\psi(z_2)|^2}{|\alpha - \phi(z_1) \psi(z_2)|^2}
\end{align*}
is continuous on the closed unit disc. However, for a general inner function $\phi$, its non-tangential limits need only exist $m$-almost everywhere on $\mathbb{T}$. Even if they do exist on the entire unit circle, $\phi^*$ need not be continuous. For this reason, we introduce the function $\Psi_r(z) := \phi(rz_1)\psi(z_2) $ for $0<r<1$. This is not an inner function, as $|\phi(r z_1)|<1$ on the unit circle. However, since $\Psi_r \to \Psi$ as $r \to 1^-$, we can investigate the Clark measures of $\Psi$ via $\Psi_r$.

\begin{theorem}
    Let $\Psi(z_1,z_2) = \phi(z_1) \psi(z_2)$ for one-variable inner functions $\phi$ and $\psi$, such that
    \begin{enumerate}[A.]
        \item $\psi$ extends to be continuously differentiable on $\mathbb{T}$ except at a finite set of points, 
        \item the solutions to $\psi^* = \beta $ for $\beta \in \mathbb{T}$ can be parameterized by functions $\{ g_k( \beta) \}_{k \geq 1}$ which are continuous in $\beta$ on $\mathbb{T}$ except at a finite set of points, 
        \item for every $\beta \in \mathbb{T}$, there are no solutions to $\psi^* = \beta$ with infinite multiplicity, and 
        \item the Clark measures of $\psi$ are all discrete. 
    \end{enumerate}
    Then the Clark measures of $\Psi$ satisfy
    \begin{align*}
        \int_{\mathbb{T}^2} f(\xi) d\sigma_\alpha( \xi) = \sum_{k \geq 1} \int_{\mathbb{T}} f(\zeta, g_k ( \alpha \overline{\phi^* ( \zeta)}) ) \frac{dm(\zeta)}{|\psi'(g_k ( \alpha \overline{\phi^* ( \zeta)}))|}
    \end{align*}
    for $f \in C(\mathbb{T}^2)$.
    \label{prod_formula1}
\end{theorem}

\begin{remark}
    The assumptions A-D are most likely excessive, but we impose them here to get an easy guarantee that the right-hand side is finite and integrable. Nevertheless, we will see some interesting examples of product functions and their Clark measures for which Theorem \ref{prod_formula1} can be applied, e.g. when $\phi(z_1) = \exp \bigl(-\frac{1+z_1}{1-z_1} \bigr)$.
\end{remark}

\begin{remark}
     Observe that there exist examples of inner functions where the Clark measure $\sigma_\alpha$ is discrete for one specific $\alpha$-value but $\sigma_\beta$ is singular continuous for $\beta \in \mathbb{T} \setminus \{ \alpha \}$, and vice versa. See Example 1 and 2 in \cite{discrete}.
\end{remark}

\begin{proof}
    Define $\Psi_r(z_1,z_2) := \phi(r z_1)\psi(z_2)$ for $0<r<1$, and note that for each $z \in \mathbb{D}^2$,
    \begin{align*}
        \frac{1-|\Psi_r(z_1,z_2)|^2}{|\alpha - \Psi_r(z_1,z_2)|^2} \to \frac{1-|\Psi(z_1,z_2)|^2}{|\alpha - \Psi(z_1,z_2)|^2} = \int_{\mathbb{T}^2} P_z(\xi) d\sigma_\alpha ( \xi)
    \end{align*}
    as $r \to 1^-$. Define, for fixed $z_2 \in \mathbb{D}$ and fixed $0<r<1$, 
\begin{align*}
    u_{z_2}^r(z_1) := \frac{1- |\psi(z_2)\phi(rz_1)|^2}{|\alpha - \psi(z_2) \phi(rz_1)|^2}, \quad z_1 \in \mathbb{D}.
\end{align*}
As $\phi(rz_1)$ is continuous and satisfies $|\phi(rz_1)|<1$ on the unit circle, $u^r_{z_2}$ is continuous on $\overline{\mathbb{D}}$. Moreover, even though $\Psi_r$ is not an inner function, it holds that 
$$ \frac{1-|\Psi_r|^2}{|\alpha - \Psi_r|^2} = \Re \bigg( \frac{\alpha + \Psi_r}{\alpha- \Psi_r} \bigg)
$$
where $( \alpha + \Psi_r)/(\alpha - \Psi_r)$ is analytic on $\mathbb{D}^2$. Hence, the left-hand side is pluriharmonic in $\mathbb{D}^2$, which in turn implies that $u^r_{z_2}$ is harmonic in $\mathbb{D}$. By the Poisson integral formula, 
\begin{align*}
     \frac{1- |\psi(z_2)\phi(z_1)|^2}{|\alpha - \psi(z_2) \phi(z_1)|^2} = \lim_{r \to 1^-} u_{z_2}^r(z_1) = \lim_{r \to 1^-}\int_{\mathbb{T}} u^r_{z_2}(\zeta) P_{z_1}(\zeta) dm(\zeta).
\end{align*}
Observe that $\Re ( ( \alpha + \Psi_r(z_1,z_2))/(\alpha - \Psi_r(z_1,z_2) ) )$ is bounded for every $(z_1, z_2) \in \overline{\mathbb{D}} \times \mathbb{D}$ and every $0 < r < 1$. The dominated convergence theorem then states that we can move the limit into the integral: so, for fixed $z_2 \in \mathbb{D}$,
\begin{align}
     \frac{1- |\psi(z_2)\phi(z_1)|^2}{|\alpha - \psi(z_2) \phi(z_1)|^2} = \int_{\mathbb{T}} \lim_{r \to 1^-} u^r_{z_2}(\zeta) P_{z_1}(\zeta) dm(\zeta).
     \label{limit of u integral}
\end{align}
Moreover,
\begin{align*}
    \lim_{r \to 1^-} u^r_{z_2}(\zeta) = \lim_{r \to 1^-} \frac{1- |\psi(z_2)\phi(r \zeta)|^2}{|\alpha - \psi(z_2) \phi(r \zeta )|^2} = \frac{1- |\psi(z_2)\phi^*(\zeta)|^2}{|\alpha - \psi(z_2) \phi^*(\zeta )|^2} 
\end{align*}
for $m$-almost every $\zeta \in \mathbb{T}$. Let $E$ denote the set of points $\zeta$ such that $|\phi^*(\zeta)|=1$.  

By our assumptions, the solutions to $\psi^* = \beta$ can be parameterized by functions $g_k( \beta)$ continuous on $\mathbb{T}$ except on a finite collection of points. Since we have also assumed that the Clark measures of $\psi$ consist of point masses, by Proposition \ref{one_var pointmass}, the measure associated to any $\beta \in \mathbb{T}$ is given by $\sum_{k\geq 1} | \psi'( g_k ( \beta))|^{-1} \delta_{g_k(\beta)}$ where $|\psi'(g_k(\beta))| > 0$ for each $k$. For fixed $\zeta \in E$, this holds for $\beta = \alpha \overline{\phi^*( \zeta)}$. 

Hence, for $\zeta \in E$, we have that
\begin{align}
    \frac{1- |\psi(z_2)\phi^*(\zeta)|^2}{|\alpha - \psi(z_2) \phi^*(\zeta )|^2} 
    = \sum_{k \geq 1} \frac{1}{| \psi'(g_k(\alpha \overline{\phi^*(\zeta)}))|} P_{z_2}(g_k(\alpha \overline{\phi^*(\zeta)})).
    \label{can we integrate}
\end{align}
To apply the Poisson integral formula, we must first check that the product of the right-hand side with $P_{z_1}(\zeta)$ is integrable. Recall that by Fatou's theorem, $\phi(r \zeta)$ converges to $\phi^*(\zeta)$ as $r \to 1^-$ $m$-almost everywhere on $\mathbb{T}$ and in $L^1(\mathbb{T})$. Moreover, the curves $\{ g_k \}_{k \geq 1} $ are assumed to be continuous on the unit circle except at finitely many points. Hence, the composition $P_{z} (\zeta, g_k (\alpha \overline{\phi^*(\zeta)}))$ must be measurable --- indeed, $f(\zeta, g_k (\alpha \overline{\phi^*(\zeta)}))$ is measurable for any $f\in C(\mathbb{T}^2)$. Similarly, we see that the weights $|\psi'(g_k ( \alpha \overline{\phi^* ( \zeta)}))|$ are measurable, as $\psi$ is assumed to be continuously differentiable on $\mathbb{T}$ except at finitely many points. Since we are integrating over a compact space, this is enough to ensure integrability. 


Moreover, for fixed $\zeta \in E $, the sum $\sum_{k \geq 1} |\psi'(g_k ( \alpha \overline{\phi^* ( \zeta)}))|^{-1}$ must be finite, since the Clark measure of $\psi$ associated to the parameter value $\alpha \overline{\phi^* ( \zeta)}$ exists by assumption. As we have excluded the situation where infinitely many of the curves intersect, the weights cannot sum up to infinity as we integrate over $\mathbb{T}$. The curves could still have infinite intersections at limit points of $g_k( \alpha \overline{\phi^*(\zeta)})$, which per definition do not solve $\psi^* =  \alpha \overline{\phi^*(\zeta)}$. However, by Proposition \ref{one_var pointmass}, the weights of the Clark measures must be zero for these points.

Since equation \eqref{can we integrate} holds for $m$-almost every $\zeta \in \mathbb{T}$, the integrals of the left- and right-hand side will coincide. By combining this with \eqref{limit of u integral}, we see that
\begin{align*}
    \frac{1- |\psi(z_2)\phi(z_1)|^2}{|\alpha - \psi(z_2) \phi(z_1)|^2} &= \int_{\mathbb{T}} \frac{1- |\psi(z_2)\phi^*(\zeta)|^2}{|\alpha - \psi(z_2) \phi^*(\zeta )|^2}P_{z_1}(\zeta) dm(\zeta)  \\
    &= \int_{\mathbb{T}}\sum_{k \geq 1} \frac{1}{| \psi'(g_k(\alpha \overline{\phi^*(\zeta)}))|} P_{z}(\zeta, g_k(\alpha \overline{\phi^*(\zeta)})) dm (\zeta).
\end{align*}
As the summands are all positive, we may interchange summation and integration. Thus,
\begin{align*}
    \frac{1- |\Psi(z_1,z_2)|^2}{|\alpha - \Psi(z_1,z_2)|^2}  = \sum_{k \geq 1}  \int_{\mathbb{T}}  P_{z}(\zeta,g_k(\alpha \overline{\phi^*(\zeta)})) \frac{dm(\zeta)}{| \psi'(g_k(\alpha \overline{\phi^*(\zeta)}))|}, 
\end{align*}
i.e. 
\begin{align*}
     \int_{\mathbb{T}^2} P_z(\xi) d \sigma_\alpha ( \xi)  = \sum_{k \geq 1}  \int_{\mathbb{T}}  P_{z}(\zeta,g_k(\alpha \overline{\phi^*(\zeta)})) \frac{dm(\zeta)}{| \psi'(g_k(\alpha \overline{\phi^*(\zeta)}))|}.
\end{align*}
Since the span of Poisson kernels is dense in $C(\mathbb{T}^2)$, we may conclude that
\begin{align*}
        \int_{\mathbb{T}^2} f(\xi) d\sigma_\alpha( \xi) = \sum_{k \geq 1} \int_{\mathbb{T}} f(\zeta, g_k ( \alpha \overline{\phi^* ( \zeta)}) ) \frac{dm(\zeta)}{|\psi'(g_k ( \alpha \overline{\phi^* ( \zeta)}))|}
    \end{align*}
for all $f \in C(\mathbb{T}^2)$.
\end{proof}

Note that the weights of these measures strongly resemble their RIF counterparts from Theorem \ref{generic}. Moreover, as in the case of the multiplicative embeddings, Theorem \ref{prod_formula1} allows for infinite collections of parameterizing functions.

\begin{remark} Let us convince ourselves that there actually exist inner functions $\psi$ that meet the requirements of Theorem \ref{prod_formula1}. For example, finite Blaschke products define one such class. Let $\psi$ be a non-constant finite Blaschke product of order $n$. As in Example \ref{ex:blaschke}, this implies that $\psi$ is analytic on $\mathbb{T}$ and $\psi^*(\zeta) = \beta $ has precisely $n$ distinct solutions for each $\beta \in \mathbb{T}$, and $\psi' \neq 0$ on $\mathbb{T}$. By the Implicit Function Theorem, we may thus parameterize the solutions with functions $\{g_k( \beta) \}_{k=1}^n$ analytic on the unit circle. Additionally, we saw in Example \ref{ex:blaschke} that the Clark measures of $\psi$ are discrete for every $\beta \in \mathbb{T}$. Hence, Theorem \ref{prod_formula1} works for any product function $\Psi(z) = \phi(z_1) \psi(z_2)$ where $\psi$ is a non-constant finite Blaschke product and $\phi$ is an arbitrary inner function. 

In the case where both $\phi$ and $\psi$ are finite Blaschke products, the theorem reproduces what we know about RIFs, as
 $$\bigg| \frac{\partial \Psi}{\partial z_2}( \zeta, g_k^\alpha(\zeta)) \bigg| = |\phi(\zeta) \psi'(g_k^\alpha ( \zeta)) | = |\psi'(g_k^\alpha ( \zeta)) |$$
Then Theorem \ref{prod_formula1} reduces to Theorem \ref{generic}. 
\end{remark}

Observe that if $\psi \in C(\mathbb{T})$, it must be a finite Blaschke product (Corollary 4.2, \cite{survey}). Similarly, if $\psi' \in H^1( \mathbb{T})$, then $\psi$ is continuous on $\mathbb{T}$ (Theorem 3.11, \cite{duren}) and thus a finite Blaschke product. Hence, to be able to construct varied examples, we need $\psi^*$ to have some discontinuities on the unit circle (see e.g. Example \ref{ex: exp product}).

In what comes next, we let $g_k^\alpha(\zeta):=g_k(\alpha \overline{\phi^*(\zeta)})$ for ease of notation. 
\begin{example}
\label{ex: blaschke exp product}
Let 
\begin{align*} \Psi(z_1, z_2) := \psi (z_2) \phi(z_1) = z_2\frac{\lambda-z_2}{1-\overline{\lambda}z_2} \exp \biggl(-\frac{1+z_1}{1-z_1}\biggr)
\end{align*}
for $\phi$ as in Example \ref{exp_ex} and some constant $\lambda \in \mathbb{D}$. Note that $\Psi^* = 0$ for $\zeta_1 = 1$. The equation $\Psi^* = \alpha$ for $\alpha \in \mathbb{T}$ can be rewritten as 
\begin{align*} \zeta_2\frac{\lambda-\zeta_2}{1-\overline{\lambda}\zeta_2} = \alpha \exp \biggl(\frac{1+\zeta_1}{1-\zeta_1}\biggr).
\end{align*}
For $\alpha = e^{i \nu}$, the solutions to this are given by $\zeta_2 = g_k^\alpha(\zeta_1)$, $k=1,2$, where 
\begin{align*}
        g^\alpha_k( \zeta_1) &:= \frac{1}{2} \biggl(\lambda + \exp \biggl( i \nu + \frac{1+\zeta_1}{1-\zeta_1} \biggr) \overline{\lambda} \\ &\pm \sqrt{-4 \exp \biggl( i \nu + \frac{1+\zeta_1}{1-\zeta_1} \biggr) + \biggl(- \lambda- \exp \biggl( i \nu + \frac{1+\zeta_1}{1-\zeta_1} \biggr)  \overline{\lambda} \biggr)^2} \, \biggr).
    \end{align*}
In Figure \ref{fig:blaschke_exp}, we have plotted the level curves for certain parameter values.
\begin{figure}
    \centering
    \includegraphics[width=7cm]{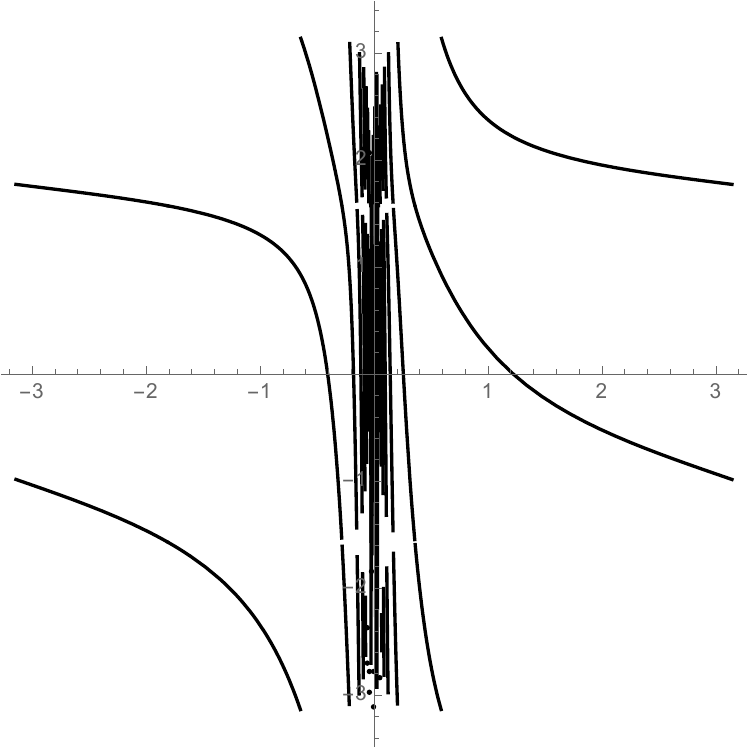}
    \caption{Level curves $g_k^\alpha$ in Example \ref{ex: blaschke exp product} for $\alpha = e^{i\pi/4}$ and $\lambda = i/2$.}
    \label{fig:blaschke_exp}
\end{figure}

Let us calculate the weights of the Clark measures. Observe that 
    $$ \psi'(z_2) = \frac{\lambda-2z_2+z_2^2\overline{\lambda}}{(1-\overline{\lambda}z_2)^2}.
    $$
    Hence, by Theorem \ref{prod_formula1}, 
    \begin{align*}
        \int_{\mathbb{T}^2} f(\xi) d \sigma_\alpha ( \xi) = \sum_{k=1}^2 \int_{\mathbb{T}} f(\zeta, g_k^\alpha( \zeta)) \frac{|1-\overline{\lambda}g_k^\alpha(\zeta)|^2}{|\lambda-2g_k^\alpha(\zeta)+g_k^\alpha(\zeta)^2\overline{\lambda}|}dm(\zeta)
    \end{align*}
    for all $f \in C(\mathbb{T}^2)$.
In Figure \ref{fig:weight_curves}, we have plotted the weights 
$$ W^{\alpha}_k(\zeta) := \frac{|1-\overline{\lambda}g_k^\alpha(\zeta)|^2}{|\lambda-2g_k^\alpha(\zeta)+g_k^\alpha(\zeta)^2\overline{\lambda}|}.
$$
\begin{figure}
    \centering
    \includegraphics[width=7cm]{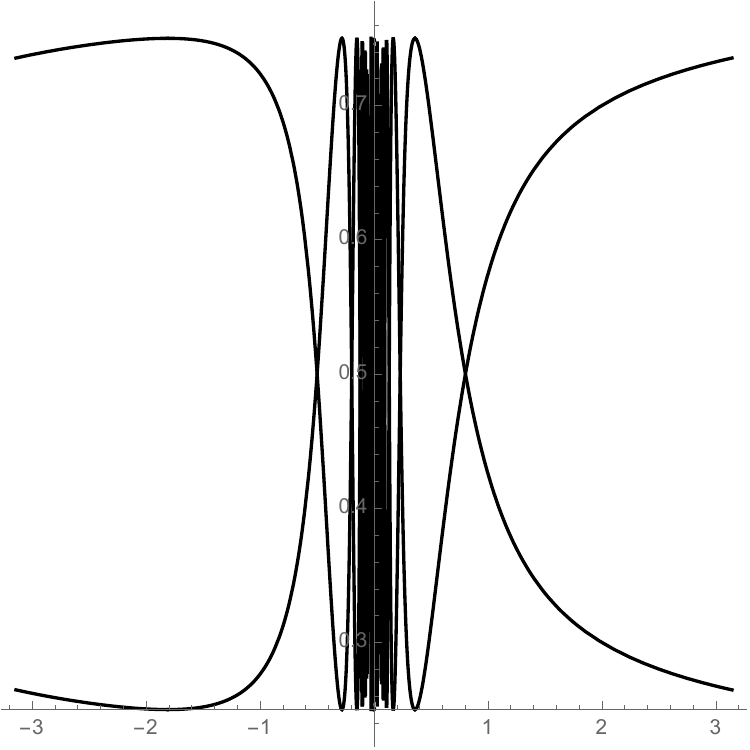}
    \caption{Weight curves $W^\alpha_k$ in Example \ref{ex: blaschke exp product} for $\alpha = e^{i\pi/4}$ and $\lambda = i/2$.}
    \label{fig:weight_curves}
\end{figure} 
\end{example}

\begin{example}
\label{ex: exp product} 
Define
$$ \Psi(z_1, z_2) := \phi(z_1) \phi(z_2) = \exp \biggl(-\frac{1+z_1}{1-z_1}\biggr) \exp \biggl(-\frac{1+z_2}{1-z_2}\biggr),
$$
where $\phi$ again is as in Example \ref{exp_ex}. As $\phi^*$ exists everywhere on $\mathbb{T}$, we have $\Psi^*(\zeta) = \phi^*(\zeta_1) \phi^*( \zeta_2)$. On the lines $\{ (1,\zeta_2) : \zeta_2 \in \mathbb{T} \} $ and $\{ (\zeta_1, 1) : \zeta_1 \in \mathbb{T} \} $ in $\mathbb{T}^2$, we see that $\Psi^* = 0$. Otherwise, $|\Psi^*| = 1$.

Since $\Psi^*$ is well-defined and unimodular on $\mathbb{T}^2$ except on the lines $\{ \zeta_1=1 \} \cup \{ \zeta_2= 1 \}$ where $\Psi^* = 0$, we need to solve the equation $\Psi = \alpha$. We may view this as 
\begin{align*}
    \exp \biggl(-\frac{1+\zeta_1}{1-\zeta_1}-\frac{1+\zeta_2}{1-\zeta_2}\biggr) = e^{i( \nu + 2 \pi k)}, \quad k \in \mathbb{Z},
\end{align*}
i.e. 
\begin{align*}
    -\frac{1+\zeta_1}{1-\zeta_1}-\frac{1+\zeta_2}{1-\zeta_2} = i(\nu + 2 \pi k), \quad k \in \mathbb{Z}.
\end{align*}
Solving for $\zeta_2$ yields
\begin{align*}
    \zeta_2 = g^\alpha_k( \zeta_1) := \frac{\nu (\zeta_1 - 1) + 2 \pi k( \zeta_1 -1) + 2 i }{\nu( \zeta_1 - 1) + 2 \pi k (\zeta_1 - 1) + 2 i \zeta_1}, \quad k \in \mathbb{Z}. 
\end{align*}
Note that functions $g^\alpha_k$ are continuous on the unit circle; their only singularities occur at points $\zeta_1 = \frac{2\pi k+\nu}{\nu + 2 \pi k + 2i }$, which do not have modulus one.

Moreover, all $g^\alpha_k$ pass through the point $(1,1) \in \mathbb{T}^2$, which does not solve $\Psi^* = \alpha$ as $\Psi^*(1,1) = 0$. However, since $\mathcal{C}_\alpha ( \Psi)$ is closed, the point $(1,1)$ nevertheless lies in the unimodular level set. Hence, 
\begin{align*}
    \mathcal{C}_\alpha ( \Psi) = \bigcup_{k \in \mathbb{Z}} \{( \zeta, g^\alpha_k(\zeta)): \zeta \in \mathbb{T} \}
\end{align*}
where $g^\alpha_k$ is analytic on $\mathbb{T}$ for every $k$. We have plotted some of these curves in Figure \ref{fig:levelcurves_expexp}.
\begin{figure}
    \centering
    \includegraphics[width=7cm]{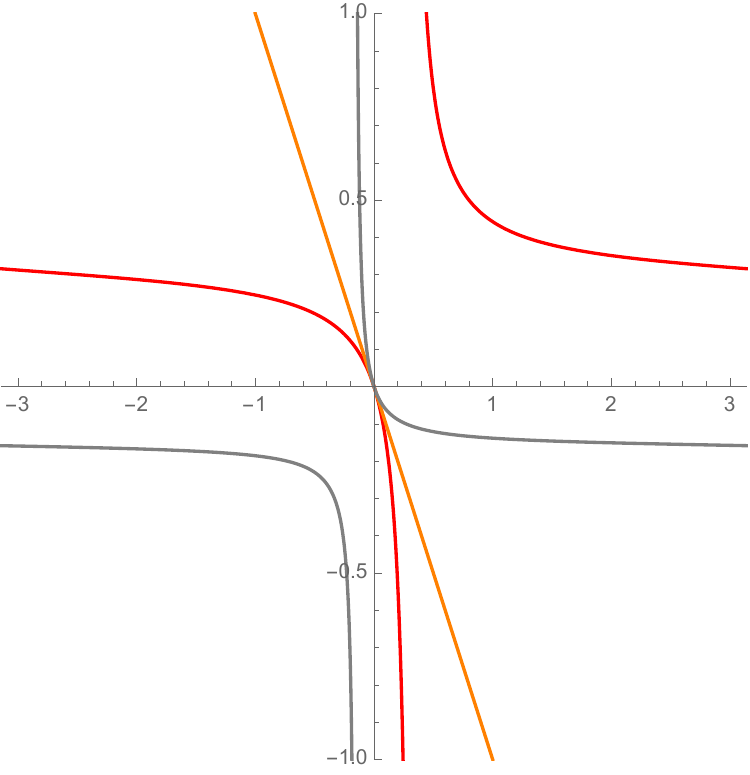}
    \caption{Level curves $g^1_k$ in Example \ref{ex: exp product} for $k=-1$ (red), $k=0$ (orange) and $k=2$ (gray).}
    \label{fig:levelcurves_expexp}
\end{figure}

Recall that by Lemma \ref{parameter}, the unimodular level sets of RIFs can be parameterized by graphs that are analytic on $\mathbb{T}^2$ except possibly at a single point. One might then expect that the Clark measures of a product function which is rational in at least one variable, like in Example \ref{ex: blaschke exp product}, would be supported on smoother curves than this $\Psi$. However, we see that in this case, the unimodular level sets are actually parameterized by much more “well-behaved” curves than in our previous example. 

At first sight, $\Psi$ does not seem to meet the requirements of Theorem \ref{prod_formula1}; there is a point on $\mathbb{T}^2$ where all $g_k$ intersect, as $g_k(1) = 1 $ for all $k \in \mathbb{Z}$. However, as noted above, this value does not in fact solve the equation $\Psi^* = \alpha$ since $\phi^*(1) = 0$. This point would cause a problem if the Clark measure of $\phi$ had positive weight there. Fortunately, we are saved by Proposition \ref{one_var pointmass}; the measure associated to $\alpha$ has a point mass at $1$ if and only if $\phi^* (1) = \alpha$, and so $|\phi'(1)|^{-1} = 0$. 

Let us now calculate the weights of the Clark measures associated to $\Psi$. First note that 
\begin{align*}
    \phi'(z_2) = - \frac{2 \exp \bigr(- \frac{1+z_2}{1-z_2} \bigl) }{(1-z_2)^2} = -\frac{2 \phi(z_2)}{(1-z_2)^2}.
\end{align*}
Then
\begin{align*}
    \phi'(g_k^\alpha(\zeta_1)) = -\frac{2 \alpha}{\phi(\zeta_1) (1-g_k^\alpha(\zeta_1))^2} = \frac{2 \alpha}{\phi(\zeta_1)} \frac{(\nu (\zeta_1 - 1)+ 2 \pi k( \zeta_1 - 1) + 2i\zeta_1)^2}{4(\zeta_1-1)^2}
\end{align*}
for $\zeta_1 \in \mathbb{T} \setminus \{ 1 \}$. When taking moduli, we find 
\begin{align*}
    |\phi'(g_k^\alpha(\zeta_1))| =  \bigg| \frac{\nu (\zeta_1 - 1) + 2 \pi k (\zeta_1 - 1) + 2i\zeta_1}{2(\zeta_1-1)} \bigg|^2
\end{align*}
for $\zeta_1 \in \mathbb{T} \setminus \{ 1 \}$. Hence, Theorem \ref{prod_formula1} yields
\begin{align*}
    \int_{\mathbb{T}^2} f( \xi) d \sigma_\alpha ( \xi ) = \sum_{k \in \mathbb{Z}}\int_{\mathbb{T}} f( \zeta, g_k^\alpha(\zeta))  \bigg| \frac{2(\zeta-1)}{\nu (\zeta - 1) + 2 \pi k (\zeta - 1) + 2i\zeta} \bigg|^2 dm(\zeta)
    \label{integral_exp}
\end{align*}
for all $f \in C(\mathbb{T}^2)$, where $\alpha = e^{i \nu}$. Since $\sum_{k \in \mathbb{Z}} \frac{1}{k^2}$ converges, we see that the right-hand side is finite. 

Note that the weights 
$$ W^\alpha_k(\zeta) := \bigg| \frac{2(\zeta-1)}{\nu (\zeta - 1) + 2 \pi k (\zeta - 1) + 2i\zeta} \bigg|^2 
$$
reduce to zero for $\zeta = 1$, as expected. Moreover, we established earlier that all the level curves pass through the singularity $(1,1)$. Based on this example, it seems that the weights “detect” the singularities of $\Psi$ --- much like in the case of the rational inner functions in Section \ref{sec:rifs}. Recall our brief discussion on the connection between the order of vanishing of weights at RIF singularities and contact order on page 9. It might be interesting to study if the singularities of general product functions are connected to the density of their Clark measures in some similar way.

\begin{figure}
    \centering
    \includegraphics[width=7cm]{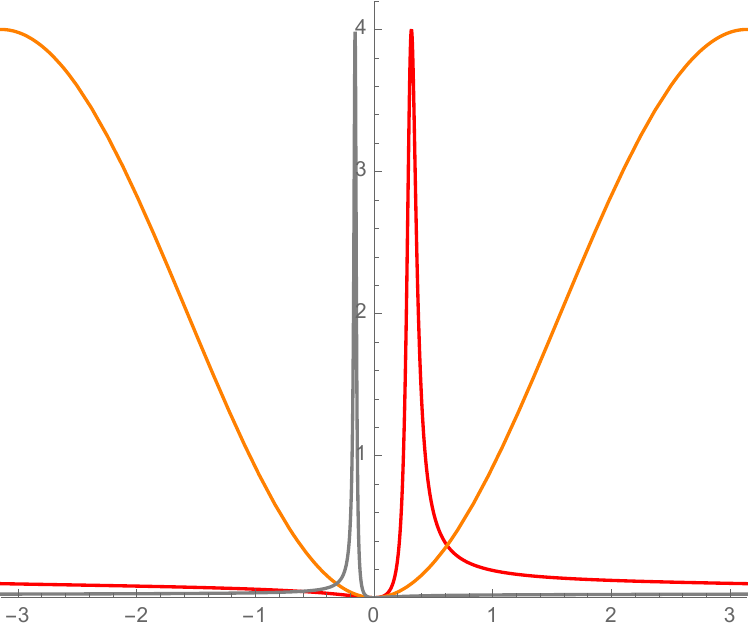}
    \caption{Weight curves $W_k^1$ in Example \ref{ex: exp product} for $k=-1$ (red), $k=0$ (orange) and $k=2$ (gray).}
    \label{fig:weight_curves2}
\end{figure} 
\end{example}

    

\section{Closing remarks}
\label{sec: close}
    It is important to note that Clark measures of general bivariate inner functions still remain unexplored. In one variable, any singular probability measure on $\mathbb{T}$ defines the Clark measure of some inner function (pp. 234-235, \cite{Intro}). In several variables, we need added requirements on a measure for it to be a Clark measure --- as discussed in Remark \ref{skeptic}, any positive, pluriharmonic, singular probability measure defines the Clark measure of some inner function. The distinction arises from the fact that in several variables, it is not as easy to ensure that a given harmonic function is the real part of an analytic function. By Theorem 2.4.1 in \cite{Rudin}, the Poisson integral of a real measure $\mu$ on $\mathbb{T}^d$ is given by the real part of an analytic function if and only if its Fourier coefficients satisfy $\hat{\mu}(k) = 0$ for every $k$ outside the set $-\mathbb{Z}_{+}^d \cup \mathbb{Z}_{+}^d$, where $-\mathbb{Z}_{+}^d$ denotes the set of points $(k_1, \ldots, k_d) $ where every $k_j \leq 0$. 

    Furthermore, the kind of smooth curve-parameterizations that were obtained for the classes of inner functions in this text are certainly not applicable for general inner functions. What we do know is that RP-measures cannot be supported on sets of Hausdorff dimension less than one (Theorem 4, \cite{bergqvist}). In two dimensions, we have seen examples of Clark measures supported on curves (i.e. sets of Hausdorff dimension one). In \cite{mcdonald}, the author constructs an RP-measure whose support has Hausdorff dimension two. However, it is not clear to the author how one would construct an RP-measure with support of Hausdorff dimension $1<d<2$. For an in-depth discussion about the supports of RP-measures, see \cite{bergqvist}. 

    We end with a brief note on Clark embedding operators associated to the classes of inner functions introduced here. In Example 4.2 in \cite{Doubtsov}, it is shown that all $T_\alpha$ are unitary for the simple multiplicative embedding $\phi(z_1 z_2) = z_1 z_2$ where $\phi(z) = z$, for which the Clark measure $\sigma_\alpha$ satisfies
    \begin{align*}
        \int_{\mathbb{T}^2} f(\xi) d \sigma_\alpha ( \xi ) = \int_\mathbb{T} f( \zeta, \alpha \overline{\zeta}) dm(\zeta), \quad f \in C(\mathbb{T}^2).
    \end{align*}
    For holomorphic monomials $f$, the functions $f(\zeta, \alpha \overline{\zeta})$ are dense in $L^2(m)$, which in turn implies that $A(\mathbb{D}^2)$ is dense in $L^2(\sigma_\alpha)$, as desired. It seems plausible that a similar argument can be applied to show that given any $\Phi(z)$ satisfying the conditions of Corollary \ref{d-dim cor}, the associated Clark embedding operators are all unitary. In the case of product functions, however, it is not so clear when the operators would be unitary and further analysis is required. 
    
    
\section*{Acknowledgements}
The author would like to express her deepest gratitude to Alan Sola, for insightful comments and expert advice. 

This material has been adapted from the author's Master's thesis in mathematics at Stockholm University in August 2023. 



\end{document}